\documentclass[a4paper]{amsart}
\usepackage[latin1]{inputenc}
\usepackage[T1]{fontenc}
\DeclareFontFamily{T1}{wncyr}{}
\DeclareFontShape{T1}{wncyr}{m}{n}{%
  <5><6><7><8><9>gen*wncyr%
  <10><10.95><12><14.4><17.28><20.74><24.88>wncyr10}{}

\usepackage{calc}	    	% Faire des calculs sur les longueurs.
\usepackage{multicol}		% Plusieurs colonnes.
\usepackage{amssymb}	% Des symboles.
\usepackage{wasysym}
\usepackage{verbatim}	% Pour les codes sources en informatique.
\usepackage{mathrsfs}	% Des lettres majuscules cursives (\mathscr).
\usepackage{dsfont}
\usepackage[dvipsnames]{xcolor}
\usepackage{pgf}
\usepackage{pgfpages}
\usepackage{tikz}
\usetikzlibrary{trees,shapes,%
  matrix,arrows,decorations.pathmorphing,backgrounds,positioning,fit}
\usepackage{indentfirst}
\usepackage{cancel}
\usepackage{fixltx2e} %%% pour avoir des \textsubscript
\usepackage{hyperref}
\theoremstyle{plain}

\newtheorem{thm}{Theorem}[section]
\newtheorem{lem}[thm]{Lemma}
\newtheorem{prop}[thm]{Proposition}
\newtheorem{coro}[thm]{Corollary}

\theoremstyle{definition}

\newtheorem{defn}[thm]{Definition}

\newtheorem{conj}{Conjecture}

\theoremstyle{remark}
\newtheorem{rem}[thm]{Remark}

\newcommand{\on}{\operatorname}

\newcommand{\mc}{\mathcal}

\newcommand{\mbb}{\mathbb}
\newcommand{\mf}{\mathfrak}

\newcommand{\mscr}{\mathscr}

\newcommand{\id}{\ensuremath{\mathop{\rm id\,}\nolimits}}

\newcommand{\Z}{\mathbb{Z}}

\newcommand{\Q}{\mathbb{Q}}

\newcommand{\F}{\mathbb{F}}
\newcommand{\A}{\mathbb{A}}

\newcommand{\Gm}{{\mbb{G}_m}}
\newcommand{\m}{\mc{M}}
\newcommand{\ve}{\varepsilon}

\newcommand{\Gal}{\on{Gal}}

\newcommand{\ol}{\overline}

\newcommand{\st}{\stackrel}
\newcommand{\mx}{\mbox}
\newcommand{\geqs}{\geqslant}
\newcommand{\leqs}{\leqslant}

\newcommand{\sm}{\setminus}
\newcommand{\ra}{\rightarrow}
\newcommand{\lra}{\longrightarrow}

\newcommand{\PSL}{\on{PSL}}

\newcommand{\GT}{\on{GT}}

\newcommand{\Perf}{\on{Perf}}
\newcommand{\Hom}{\on{Hom}}

\newcommand{\Ext}{\on{Ext}}
\newcommand{\Gr}{\on{Gr}}
\newcommand{\HH}{\on{H}}

\newcommand{\CH}{\on{CH}}
\newcommand{\p}{\mathbb{P}}

\newcommand{\Sp}{\on{Spec}}
\newcommand{\Spec}{\Sp}
\newcommand{\Pic}{\on{Pic}}
\newcommand{\Of}[1][{}]{\mscr{O}_{#1}}

\newcommand{\Supp}{\on{Supp}}
\newcommand{\dme}[1][]{\mc{DM}_{gm}^{eff}}
\newcommand{\dm}[1][]{\mc{DM}_{gm}}

\newcommand{\DM}{\on{DM}}
\newcommand{\DMT}{\on{DMT}}
\newcommand{\MTM}{\on{MTM}}

\newcommand{\Sm}[1][]{\on{Sm_{#1} } }

\newcommand{\pst}[1]{{\p^1 #1 \setminus \{0,1,\infty\}} }
\newcommand{\ps}{{\pst{}}}

\newcommand{\mob}[1][n]{{\ol{\m_{0,#1}}}}
\newcommand{\mo}[1][n]{{\m_{0,#1}}}
\newcommand{\moD}[2][n]{{\m_{0,#1}^{#2}}}
\newcounter{notemargin}
\newcommand{\marginote}[1]{
\addtocounter{notemargin}{1}%
$^{\text{\arabic{notemargin}}}$%
\marginpar{\tiny{\arabic{notemargin}.-- #1}}
}

\newcommand{\Nis}{\on{Nis}}
\newcommand{\Mod}{\on{Mod}}
\newcommand{\SH}{\on{SH}}

\newcommand{\MS}[2][\ensuremath{\mathbb P}^1]{\on{Spt}_{#1}^{\Sigma}(#2)}
\newcommand{\MZ}[2][\ensuremath{\mathbb Z}]{\mathrm{M}#1_{#2}}
\newcommand{\DMZ}[2][\ensuremath{\mathbb Z}]{\on{DM}_{#1}(#2)}
\newcommand{\Ssp}[3][\ensuremath{\p^1}]{\Sigma_{#1}^{#3} ({#2}_{+})}

\newcommand{\Sus}[2]{\Sigma^{#1,#2}}
\newcommand{\ots}[1][]{\otimes_{#1}}
\newcommand{\Bl}{\on{Bl}}

\newlength{\ledge}
\newlength{\sibdis}
\newlength{\ecan}
\newlength{\ecas}
\newlength{\ecae}
\newlength{\ecao}
\newlength{\eca}
\newlength{\lbullet}
\tikzset{deftree/.style={level distance=\ledge,sibling distance=\sibdis}}
\tikzset{edgesp/.style={level distance=#1*\ledge,sibling distance=#1\sibdis}}
\tikzset{labf/.style={mathsc,yshift=-\eca}}
\tikzset{labfs/.style={mathss,yshift=-\eca}}
\tikzset{labr/.style={mathsc,yshift=\eca}}
\tikzset{labrs/.style={mathss,yshift=\eca}}
\tikzset{math mode/.style = {execute at begin node=$, execute at end node=$}}
\tikzset{mathscript mode/.style =%
 {execute at begin node=$\scriptstyle , execute at end node=$}}
\tikzset{math/.style = {execute at begin node=$, execute at end node=$}}
\tikzset{mathsc/.style =%
 {execute at begin node=$\scriptstyle , execute at end node=$}}
\tikzset{mathss/.style =%
 {execute at begin node=$\scriptscriptstyle , execute at end node=$}}
\tikzset{root/.style={draw,circle,inner sep=1pt,execute at begin node=$\bullet,
    execute at end node=$}}
\tikzset{roottest/.style={draw,circle,inner sep=#1pt}}
\tikzset{roots/.style={draw,circle,inner sep=2pt}}
\tikzset{bull/.style={fill,circle,minimum size=2pt,inner sep=0pt}}
\tikzset{leaf/.style={minimum size=2pt,inner sep=1pt}}
\tikzset{leafb/.style={minimum size=0pt,inner sep=0pt}}
\tikzset{intvertex/.style={mathsc,fill,circle,minimum size=0.6ex, inner sep=0pt}}
\tikzset{Reda/.style={-,double distance=0.3ex, draw=black}}
\tikzset{N/.style={-,thin,draw=black}}
\tikzset{Nd/.style={-,dotted, thin,draw=black}}
\renewcommand{\marginpar}[1]{}
\renewcommand{\marginote}[1]{}

\title
{A motivic Grothendieck-Teichmüller Group}% over $\Spec(\Z)$}
\author{Ismael Soud{\`e}res}
\address{Universität Osnabrück \\Institut für Mathematik
\\
Albrechtstr. 28a \\
49076 Osnabrück \\
Germany \\
ismael.souderes@uni-osnabrueck.de}
\date{\today}
\begin{document}
%\begin{multicols}{2}
\thanks{}
\begin{abstract}This paper proves the Beilinson-Soulé vanishing conjecture for
  motives attached to the moduli spaces of curves of genus $0$ with $n$ marked
  points, $\m_{0,n}$. As part 
  of the proof, it is also proved that these motives are mixed Tate. As a
  consequence of Levine's work, one obtains then well defined categories of
  mixed Tate motives over the moduli spaces of curves $\m_{0,n}$. It is shown
  that morphisms between $\m_{0,n}$'s forgetting marked points and embedding as
  boundary components induce functors between those categories and how
  tangential bases points fit in these functorialities. 

Tannakian formalism
  attaches groups to these categories and morphisms reflecting the
  functorialities leading to the definition of a motivic
  Grothendieck-Teichmüller group.  

Proofs of the above properties rely on the geometry of the tower of the
$\m_{0,n}$. This allows us to treat the general case of motives over 
$\Spec(\Z)$ with $\Z$ coefficients working in Spitzweck's category 
of motives. From there, passing to $\Q$-coefficients we deal with the classical
tannakian formalism and explain how working over $\Spec(\Q)$ allows
a more concrete description of the tannakian group.  
 \end{abstract}
\maketitle
%%%%%%%%%%%%%%%%%%%%%%%%%%%%%%%%%%%%%%%%%%%%%%%%%%%%%%%%%%%%%%%%%%%%
%%%%%%%%%%%%%%%%%%%%%%%%%%%%%%%%%%%%%%%%%%%%%%%%%%%%%%%%%%%%%%%%%%%%
%%%%%%%%%%%%%%%%%%%%%%%%%%%%%%%%%%%%%%%%%%%%%%%%%%%%%%%%%%%%%%%%%%%%
%%%%%%%%%%%%%%%%%%%%%%%%%%%%%%%%%%%%%%%%%%%%%%%%%%%%%%%%%%%%%%%%%%%%
%%%%%%%%%%%%%%%%%%%%%%%%%%%%%%%%%%%%%%%%%%%%%%%%%%%%%%%%%%%%%%%%%%%%
%%%%%%%%%%%%%%%%%%%%%%%%%%%%%%%%%%%%%%%%%%%%%%%%%%%%%%%%%%%%%%%%%%%%
\tableofcontents
\section{Introduction}
In \cite{LEVTMFG}, M. Levine considers a smooth quasi-projective variety
$X$ over a number field $\F$. He shows that when the motive of $X$ in
$\DMZ[/\F,\Q]{\Spec(\F)}$ is mixed Tate and   satisfies the Beilinson-Soulé vanishing
properties, one has a well defined tannakian category of mixed Tate motives
$\MTM_{/\F, \Q}(X)$ whose tannakian Hopf algebra $H_X$ is built out of a complex of
algebraic cycles computing the higher Chow groups. Moreover M. Levine
proves that the tannakian group $G_X=\Spec(H_X)$ fits in a short exact sequence 
\[
1 \lra G_{X, geom} \lra G_X \lra 
G_{\Spec(\F)} \lra 1
\] 
where $G_{X, geom}$ can be identified with Deligne-Goncharov motivic fundamental
group $\pi_1^{mot}(X,x)$ after a choice of a (tangential) base point  $x \in
X(\F)$.

The above exact sequence admits a Lie coalgebra counterpart 
\[
0 \lra L^c_{\Spec(\F)}\lra L^c_X \lra L^{c}_{X,geom} \lra 0
\]
by considering the set of indecomposable elements of $H_X$. In
\cite{SouBarbase},  the author shows how explicit algebraic cycles, built in 
\cite{SouMPCC}, describe the coaction of $L^c_{\Spec(\F)}$ on $L^c_{X,geom}$ in
the case where 
\[
X=\m_{0,4}\simeq \ps.
\]

In order to generalize this work to any $\m_{0,n}$, the moduli space of curves
in genus $0$ with $n$ marked points, the first step consists in showing that
the moduli spaces $\m_{0,n}$ satisfy the Beilinson-Soulé vanishing conjecture.

However, if working over $\Spec(\F)$ allows M. Levine to relate $H_n$ to a
cycle complex computing motivic cohomology, the moduli space of curves are
well-defined over $\Spec(\Z)$. Therefore there is no reason to restrict ourselves
to $\Spec(\F)$ when considering only the Beilinson-Soulé vanishing property as
one can simply work in Cisinski-Déglise framework \cite{CiDegTCM}. Even more generally, the
Beilinson-Soulé vanishing property and the mixed Tate property are true in
M. Spitzweck framework \cite{SpitCP1S} of motives over $\Spec(\Z)$ with $\Z$
coefficients as proved at Theorem \ref{thm:M0nMTMBS}.

From Theorem \ref{thm:M0nMTMBS} giving the Beilinson-Soulé vanishing property
for the moduli spaces of curves $\m_{0,n}$, we deduce from M. Spitzweck's work
\cite{SpitCP1S, DFGTMSpit} that there exists a well defined triangulated
category $\DMT_{/\Spec(\Z), \Z}^{gm}(\m_{0,n})$ of mixed Tate motives over the
$\m_{0,n}$ (Theorem \ref{thm:DMTM0n}). There are  natural morphisms making 
 the moduli spaces $\m_{0,n}$ ``into a tower''. 
These morphisms are given by forgetting some marked points and by embedding of
$\ol{\m_{0,n_1}}\times \m_{0,n_2}$ as codimension $1$ boundary component of
$\ol{\m_{0,n_1+n_2-2}}$ on the Deligne-Mumford compactified tower.  These
morphisms induce functors between the categories $\DMT_{/\Spec(\Z),
  \Z}^{gm}(\m_{0,n})$ and morphisms between the corresponding tannakian groups
when working with 
$\Q$ coefficients. This leads to the definition of a motivic
Grothendieck-Teichmüller group at Definition \ref{GTderbQcoeff}.

The structure of the paper is as follows:
\begin{itemize}
\item Section \ref{sec:spitweckmot} reviews the framework of motivic $\p^1$
  spectra and the stable motivic homotopy category $\SH(S)$. It presents shortly
  M. Spitzweck's triangulated category of mixed motives over $S$ and reviews
  some of its properties: Gysin/localization triangle, projective bundle formula
  and blow-up formula. These  are derived from Déglise's work \cite{AGTIIDeg}
  because M. Spitzweck construction relies on an oriented  $E_{\infty}$-ring
  spectrum.
\item Section \ref{sec:m0n} reviews the geometry of the $\m_{0,n}$ and their
  Deligne-Mumford compactification $\ol{\m_{0,n}}$. It  proves
  that the triviality of normal bundle of $D_0$ in $\m_{0,n}\cup D_0$ for any
  open codimension $1$ stratum of  $\ol{\m_{0,n}}$.  Then it proves that the
    motives of  $\ol{\m_{0,n}}$ are mixed Tate
    over $\Spec(\Z)$ and satisfies the Beilinson-Soulé vanishing property. Then,
    it proves that the same holds for the open moduli spaces $\m_{0,n}$.
\item Section \ref{sec:GTmotgen} begins by reviewing the construction of limits
  motives as developed in \cite{OAMSpit, SpitMALM} and in \cite{Ayoub6OGII} and
  the case of motivic tangential based points. Then it shows how limits motives applied
  to the moduli space of curves $\m_{0,n}$ and an open codimension $1$ stratum $D_0$
  lead to natural functor between $\DMT_{/\Spec(\Z), \Z}^{gm}(\m_{0,n})$  and
  $\DMT_{/\Spec(\Z), \Z}^{gm}(\m_{0,n_1}\times \m_{0,n_2})$.  Tangential base
  points lead to functors 
\[
\DMT_{/\Spec(\Z), \Z}^{gm}(\Spec(\Z))\lra \DMT_{/\Spec(\Z), \Z}^{gm}(\m_{0,n}).
\]
Functoriality with respect to forgetful morphisms is a consequence of
M. Spitzweck construction. Working over $\Spec(\Z)$ with integral coefficients,
these categories are 
equivalent to categories of perfect representations of affine derived group
schemes. The above functorialities lead, as a conclusion of this section, to
the definition of a motivic Grothendieck-Teichmüller  group in this setting.
\item Section \ref{rationalcoefsetting} derives some consequences of the above
  constructions in more classical settings. In particular, working with
  $\Q$-coefficients, one obtains a tannakian group associated to the tannikian
  category given by the heart of the
  $t$-structure of   $\DMT_{/\Spec(\Z), \Q}^{gm}(\m_{0,n})$. This leads to a
  motivic Grothendieck-Teichmüller group defined in terms of automorphisms of groups
  (an not derived groups). Working over $\Spec(\F)$, the spectrum of a number
  field, we show how Deligne-Goncharov category of mixed Tate motives of the
  ring of its integers, agrees with M. Spitzweck construction of mixed
  motives and how our construction passes to this context. The end of the
  section presents the relation with
  M. Levine's approach to mixed Tate   motives and algebraic cycles.
\item The last section is devoted to some conjectures about the ``geometric''
  (derived) groups defining the motivic Grothendieck-Teichmüller group and
  its relations to Betti and De Rham realizations. 
\end{itemize}

\section{Short review of Spitzweck's mixed motives
  category}\label{sec:spitweckmot} 
Let $S$ be noetherian separated scheme of finite Krull dimension.  In
\cite{SpitCP1S}, Spitzweck built a $E_{\infty}$-ring  object
$\MZ{S}$ in the category $\MS{S}$ of motivic symmetric
 $\p^1$-spectra (cf. 
\cite{JardMSS, MHTDLOR, HoveySSSGMC}). The $\p^1$-spectrum $\MZ{S}$ serves in particular as
the motivic
Eilenberg-MacLane spectrum. It is also an oriented ring spectrum, that is in $\SH(S)$
it is an algebra 
over $\mathbf{MGL}$, the algebraic cobordism spectrum. Considering the category
$\Mod_{\MZ{S}}$ of modules over 
$\MZ{S}$, Spitzweck used a model structure on $\Mod_{\MZ{S}}$
compatible with the one on $\MS{S}$ and defined a triangulated category
$\DMZ{S}$ of motives 
 over $S$ with integral coefficients together with the following adjoint functor:
\[
\begin{tikzpicture}
\matrix (m) [matrix of math nodes,
 row sep=2em, column sep=2em]%, 
% text height=2.5ex,  text depth=0.25ex] 
{\Mod_{\MZ S} & \MS S  \\
};
\path[->,font=\scriptsize]
(m-1-1) edge node {} (m-1-2);
\path[->,font=\scriptsize, bend left]
(m-1-2) edge node[auto] {$\otimes\MZ{S}$} (m-1-1);
\end{tikzpicture} 
\]
and 
\[
\begin{tikzpicture}
\matrix (m) [matrix of math nodes,
 row sep=2em, column sep=2em]%, 
% text height=2.5ex,  text depth=0.25ex] 
{\DMZ{S} & \SH(S)  \\
};
\path[->,font=\scriptsize]
(m-1-1) edge node {} (m-1-2);
\path[->,font=\scriptsize, bend left]
(m-1-2) edge node[auto] {$\otimes \MZ{S}$} (m-1-1);
\end{tikzpicture} 
\]
where the left to right  functors $\lra$ are forgetful functors and the tensor
products is the one given by symmetric monoidal structure of $\MS{S}$
(corresponding to the smash product $\wedge$ in  \cite{JardMSS}).
 
We recall below some definitions and properties needed for our construction of
 motivic Grothendieck-Teichmüller group. Our construction is geometric and is
based on  the most expected distinguished triangles in $\DMZ{S}$ and on the
functoriality of its construction. In particular,  we recall below  Gysin's
distinguished triangles and Blow-ups formula in Spitzweck's category. Because of
the existence of Chern class in Spitzweck category and its relation with the
stable motivic homotopy category, these are
direct consequences of Déglise works \cite{AGTIIDeg}. Working
over a number field and with $\Q$ coefficient would cancel the need of the
following subsections as distinguished triangles where proved in
\cite{Vo00} and functoriality for the associated mixed Tate categories would
be insured by Levine's work \cite{LEVTMFG}.   

\subsection{Symmetric spectrum, $\SH(X)$ and mixed motives.}
Let $\Sm_S$ denote the category of smooth schemes of finite type  over $S$ and
$\Sm_S|_{\Nis}$ 
the smooth Nisnevich site over $S$. We recall below some facts about Spitzweck
construction \cite{SpitCP1S} and we are mostly interested in the case where $S$
is $\Sp(\Z)$. 

Let $Spc(S)$ be the category of (motivic) spaces over $S$, that is of Nisnevich
sheaves over $S$ with value into simplicial sets. M. Spitzweck construction
actually uses complexes of sheaves of abelian groups. Classical comparison
functor  and transfer of structures insure that his construction passes to
motivic spectra.  Via the Yoneda
embedding any scheme in $\Sm_S$ is a motivic space (constant in the simplicial
direction); any simplicial set is also a motivic space as a constant sheaf. The
terminal object is represented by $S$ itself.

A pointed (motivic) space is a motivic space $X$ together with a map
%\marginote{what is the point ??}
\[
x : S \lra X.
\]
The category of pointed spaces is denoted by $Spc_{\bullet}(S)$. To any 
space $X$, one associates a canonical pointed space $X_+=X\sqcup *$. The category
$Spc_{\bullet}(S)$ admits a monoidal structure $\otimes$
 induced by the smash product on pointed simplicial sets.

Recall that the simplicial circle is the coequalizer of
\[
\begin{tikzpicture}
\matrix (m)[matrix of math nodes, row sep=3em, column sep=5em, text
height=2.5ex, text depth=0.25ex]
{\Delta[0] &  \Delta[1] \\};
\node[yshift=0.5ex] (a) at (m-1-1.east){};
\node[yshift=0.5ex] (b) at (m-1-2.west){};
\node[yshift=-0.9ex] (c) at (m-1-1.east){};
\node[yshift=-0.9ex] (d) at (m-1-2.west){};
\path[->]
(a) edge  (b)
(c) edge (d);
\end{tikzpicture}
\] 
and let $S^1_s$ be the corresponding pointed space. Moreover, let $S^1_T$, the
Tate circle, be the pointed space represented by $(\p^1,\{\infty\})$.  

Very shortly, a symmetric $\p^1$-spectrum $E$ is a collection of pointed spaces
$E=(E_0,E_1,\ldots)$ with structure maps $S^1_T \otimes E_n \ra E_{n+1}$ and
with the extra data of a symmetric group actions $\Sigma_n \times E_n \ra E_n$
such that the composition maps 
\[
(S^1_T)^{\otimes p} \otimes E_n \lra E_{n+p}
\]
are $\Sigma_p\times \Sigma_n$ equivariant.

 The iterated products of $S^1_s$
(resp. $S^1_T$) are denoted by $S^n_s$ (resp. $S^m_T$).
Tensoring with the simplicial circle (resp. the Tate circle) induce a simplicial
(resp. a Tate) 
suspension functor denoted by $\Sigma_s^1$ (resp. $\Sigma_T^1$). Any motivic
space $X$ induces a symmetric $\p^1$-spectrum
\[
\Sigma_T^{\infty} X_+=(X_+, S^1_T \otimes X_+, S^2_T \otimes X_+, \ldots).
\]
We denote by $\MS S$ the category of symmetric $\p^1$-spectra. The \emph{(motivic)
  stable homotopy category} $\SH(S)$ is obtained from $\MS S$ by inverting stable weak
equivalence \cite{HoveySSSGMC}. %; that is Nisnevitch, $A^1$-homotopy invariance
%and the suspension functor $\Sigma_T^1$ . 
In particular, the suspension functors $\Sigma_s$ and $\Sigma_T$ are invertible
as are $\A^1$ weak equivalences. 

The category $\SH(S)$ is a triangulated category with shift induced by $\Sigma^1_s$.
In $\SH(S)$ the suspension functor $\Sigma_s$ will be denoted by the shift notation $[1]$. Note that $S^1_T$ is
  isomorphic to $S^1_s\ots (\Gm,\{1\})$ in $\SH(S)$.

Spitzweck, in \cite[Definition 4.27]{SpitCP1S}, defined a $\p^1$-spectrum
$\MZ{S}$ or simply 
$\MZ{}$ when $S$ is clear enough. The $\p^1$-spectrum $\MZ S$ is an
$E_{\infty}$-ring object in $\MS S $ and induces a ring 
object in $\SH(S)$ again denoted by $\MZ{S}$.

The category of motives $\DMZ{S}$ is defined as the homotopy category of modules
(in $\p^1$-spectra) over $\MZ{S}$. For any $X \in \Sm_S$, the category $\DMZ{X}$
is defined similarly as the homotopy category of modules over $f^*\MZ{S}$
 where
$f : X \ra S$ is the structural morphism and $f^* : \MS S  \ra \MS X$ the
pull-back functor between spectra categories. In the case when we need to
remember over which base $S$ we are working, we may write $\DMZ[{/S,\Z}](X)$.

For $X\st{f}{\ra} S$ in $\Sm_S$, we have a functor
%\marginote{check with Markus
%  for the suspension. In particular the $s$ infinite suspension}
\[
\begin{tikzpicture}
\matrix (m) [matrix of math nodes,
 row sep=3em, column sep=4em]%, 
% text height=2.5ex,  text depth=0.25ex] 
{\Sm_X & \DMZ X \\
Y & M_X(Y)=\Ssp{Y}{\infty}\ots f^*\MZ{S}
\\
};
\path[->,font=\scriptsize]
(m-1-1) edge node[auto]  {$M_X$} (m-1-2);
\path[|->,font=\scriptsize]
(m-2-1) edge node[auto] {} (m-2-2.west);
\end{tikzpicture} 
\] 

In $\DMZ X$, the tensor unit $f^* \MZ S $ will be denoted by $\Z_X(0)$.
The Tate object $\Z_X(1)$ is defined by 
\[
\Z_X(1)[2]= \Sigma_T^{1}\Z_X(0)=(\p^1,\infty)\ots f^*\MZ S 
\] and
corresponds as usually to the cone of the morphism
%\[
%M_X( ??) \lra M_X(\Gm)
%\]
%or 
\[
M_X(\{\infty\}) \lra M_X(\p^1)
\]
shifted by $-2$.

The suspension $\Sigma^{n-2p}_s \circ \Sigma_T^p$ will be denoted by $\Sus n p$.

\begin{rem}\label{6functorformalisme}
 M. Spitzweck shows in \cite[Section 10]{SpitCP1S} that the functor $X \lra \DMZ
 X$ satisfies the $6$ functors formalism.
\end{rem}
 
\subsection{An oriented cohomology theory}
For $f : X \ra S$ a smooth scheme over $S$, Spitzweck showed \cite[Proposition
11.1]{SpitCP1S} that the ring objects $\MZ S$ and $f^*\MZ S$ are oriented in the
sens of F. Morel and G. Vezzosi \cite{VezBSSH}, that is there is a distinguished
element
\[
\nu \in \Hom_{\SH(X)}(\Ssp[]{\p^{\infty}}{\infty}, \Sus 2 1 f^*\MZ S),
\]   
where $\p^{\infty}$ denotes the colimit of the $\p^n$, such that $\nu$ restricts
to the canonical element induced by the unit of $f^*\MZ S$ in
$\Hom_{SH(S)}(\Ssp[]{\p^1}{\infty}, \Sus 2 1 f^*\MZ S)$.

When $S$ is regular, the morphism % for the isomorphism case when $S$ is regular)
\[
\Pic(Y)\lra \Hom_{\SH(S)}(\Ssp[]{Y}{\infty},\Ssp[]{\p^{\infty}}{\infty}),
\]  
 for any $Y\ra X$ smooth, is an isomorphism and endows $\DMZ X $ with an
 orientation as described by 
 F. Déglise in \cite[2.1-(Orient) axiom]{AGTIIDeg} (see \cite[2.3.2]{AGTIIDeg} or   
\cite[Proposition 4.3.8]{MorVoeHTS}). 
That is,  for any $Y\lra X$ smooth there is an application, called \emph{first
  Chern class}:
\[
c_1 : \Pic(Y) \lra \Hom_{\DMZ X}(M_X(Y),\Z_X(1)[2])
\]
which is functorial in $Y$ and such that the image of the canonical bundle on
$\p^1_X$ is the canonical projection.

Note that the formal group law attached to the first Chern class is the
additive ones (\cite[Theorem 7.10]{SpitCP1S})

Thanks to F. Déglise's work in \cite{AGTIIDeg}, one obtains  then the following
properties:
\subsection{Distinguished triangles and split formulas in $\DMZ X$ }
Let $Y$ be a smooth scheme in $\Sm_X$ and $p: P\lra Y$ a projective bundle over
$Y$ of rank $n$. We denote by $\lambda$  the canonical line
bundle over $P$ and put 
\[
c=c_1(\lambda) : M_X(P) \lra \Z_X(1)[2]. 
\]
The diagonal $\delta_i :P \ra \underbrace{ P\times_Y \cdots \times_Y P}_{i+1\mbox{
    times}}$ followed by $p_* \ots c^{\ots i}$ gives a morphism
\[
\epsilon_{P,i} : M_X(P) \lra  M_X(Y)(i)[2i]. 
\]
\begin{prop}[Projective bundle formula {(\cite[Theorem 3.2]{AGTIIDeg})}]\label{projbundle}
With the above notation, the morphism
\[
\epsilon_{P} : M_X(P) \lra  \bigoplus_{i=1}^{n}M_X(Y)(i)[2i]
\qquad \epsilon=\sum_{i=0}^n \epsilon_{p,i}
\]
is an isomorphism.
\end{prop}

For any $0\leqs r\leqs n $, one can now define the embedding
\[
\iota_i : M_X(Y)(r)[2r]\st{*(-1)^r}\lra \bigoplus_{i=1}^{n}M_X(Y)(i)[2i] \lra M_X(P).
\]

Let $Z$ be a smooth closed subscheme  of $Y$ smooth such that $Z$ is everywhere of
codimension $n$. As in more general situations, one defines the motive of $Y$ with
support in $Z$ as 
\[
M_{X,\Supp(Z)}(Y)=M(Y/(Y\sm Z)).
\] 

Then F. Déglise in Proposition 4.3 of \cite{AGTIIDeg}, attached to the pair
$(Z,Y)$ a unique isomorphism 
%\begin{prop}[{Purity isomorphism \cite[4.3]{AGTIIDeg}}]
%There exists a unique familly of isomorphisms
\begin{equation}\label{purityZY}
\mf{p}_{Y,Z} : M_{X, \Supp(Z)}(Y) \lra M_X(Z)(n)[2n] \tag{purity}
\end{equation}
which is functorial with respect to Cartesian morphism of such pairs and such
that, when $E$ is a vector bundle over $Y$ of rank $n$ and $P=\p(E\oplus 1)$,
$\mf p_{P,X}$  
is the inverse of 
\[
M_X(Y)(n)[2n] \st{\iota_n}{\lra} M_X(P) \lra M_{X,\Supp(Y)}(P).
\]  
Defining the Thom motive $M_XTh_Y(E)$ of a rank $n$ vector bundle $E$ over $Y$,
one obtains  the Thom isomorphism \cite[\S 4.4]{AGTIIDeg}:
\[
\mf p_{E,Y} : M_XTh_Y(E) \lra M_X(n)[2n].
\]

The purity isomorphism allows us to rewrite the localization distinguished
triangle as follows.
\begin{prop}[{Gysin triangle \cite[Definition 4.6]{AGTIIDeg}}]\label{GTci}Let
  $Z$ be smooth closed subscheme  of 
  $Y$ smooth such that $Z$ is everywhere of codimension $n$. Then there is a
  distinguished triangle

\begin{equation}\label{eq:GysinT}
M_X(Y\sm Z) \st{j_*}{\lra} M_X(Y) \st{i^*}{\lra}
M_X(Z)(n)[2n]\st{\partial_{Y,Z}}{\lra}
M_X(Y\sm Z)[1]
\end{equation}
where $i^*$ (resp. $\partial_{X,Z}$) is called the \emph{Gysin morphism}
(resp. \emph{residue morphism}). 
\end{prop}
The Gysin triangle is functorial and in particular compatible with the
projective bundle isomorphisms and with the induced embeddings $\mf
l_r$. Moreover Gysin morphisms are multiplicative with respect to compositions
and products \cite[corollaries 4.33 and 4.34]{AGTIIDeg}.

Using the projective bundle formula first in the case of $Y=X$, F. Déglise showed
that $p : \p^n_X \lra X$ admits  a strong dual given by \cite[Definition 5.6]{AGTIIDeg}: 
\[
M_X(\p^n_X)(-n)[-2n]
\]
together with a Gysin morphism $p^* : M_X(X) \lra M_X(\p^n_X)(-n)[-2n]$.
This duality allows to define a Gysin morphism $p_X^* : M_X(X) \lra
M_X(\p^n_X)(-n)[-2]$ as the transpose of $p_*$. Then, generalizing this
situation to smooth projective $X$-schemes and factorizing a projective
morphism $f: Y_1 \lra Y_2$ between such as a closed immersion and a projection
$f=p\circ i$, F. Déglise obtained a Gysin morphism \cite[Definition 5.12]{AGTIIDeg}
\[
f^*:=i^*\circ p^* : M_X(Y_2) \lra M_X(Y_1)(-d)[-2d]
\]
where $i$ is of codimension $d+n$ into $\p^n_{Y_2}$

Now, let $p: Y \lra X$ be a smooth projective scheme over $X$ of pure dimension
$n$. The above Gysin morphisms lead to the 
definition of 
\[
\mu_Y :\Z_X(0) \st{p^*}{\ra} M_X(Y)(-n)[-2n] 
\st{\delta_*}{\ra} M_X(Y)(-n)[-2n] \ots M_X(Y)
\] 
and 
\[
\epsilon_Y: M_X(Y) \ots M_X(Y)(-n)[-2n] \st{\delta^*}{\ra}M_X(Y) 
\st{p_*}{\ra} \Z_ X(0)
\]
where $\delta$ is the diagonal  $\delta : Y \lra Y \times_X Y$.
\begin{thm}[Duality - {\cite[Theorem 5.23 and Proposition 5.26]{AGTIIDeg}}]
\label{Duality}Let $p:Y\lra X$ be a smooth projective scheme over $X$. The
motive 
$M_X(Y)(-n)[-2n]$, endowed with $\mu_Y$ and 
$\epsilon_Y$, is a strong dual of $M(X)$.
%
% Moreover, for $f : Y_1 \lra Y_2$ a projective morphism between
% smooth projective $S$-schemes with $Y_2$ (resp. $Y_1$ ) is of constant relative
% dimension $n$ (resp. $m$) the Gysin morphism $f^*$ is the transpose (that is the
% dual) of 
% \[
% ^t(f_*)(-n)[-2n].
% \]
\end{thm}
%---- Strong dual ?
%\begin{prop}[Gysin triangle (projective)]\label{GTpr}
%\end{prop}
%\begin{prop}[Duality]\label{duality}
%\end{prop}
\begin{prop}[{Blow-up formula - \cite[Theorem 5.38]{AGTIIDeg}}]\label{Blup}
Let $Y$ be a smooth scheme over $X$ and $Z$ a smooth closed subscheme of $Y$
purely of codimension $n$.

 Let $B_Z (Y)$ be the blow-up of $Y$  with center $Z$ and $E_Z$ be the exceptional
divisor, then:
\[
M(B_Z(Y))\simeq M_X(Y)\oplus \bigoplus_{i=1}^{n-1}M_X(Z)(i)[2i].
\]
\end{prop}

\subsection{Beilinson-Soulé's Vanishing property}
M. Levine, in \cite{LEVTMFG}, proved that if $X$ is a smooth variety over a
number field $\F$ with a motive of mixed Tate type and satisfying Beilinson-Soulé
vanishing property (see  \eqref{BS}) then there exits a well defined
tannakian (in particular $\Q$ linear) category $\MTM_{/\F, \Q}(X)$ of mixed Tate
motives over $X$ (\cite[Theorem 3.6.9]{LEVTMFG}) together with a short exact
sequence relating the tannakian groups of $\MTM_{/\F, \Q}(\F)$ and of $\MTM_{/\F, \Q}(X)$
(\cite[Section 6.6 ]{LEVTMFG}). This
short exact sequence is a motivic avatar of the short exact sequence for etale
fundamental groups relating $\Gal(\bar \F / \F)$ and $\pi_1^{et}(X)$.

In a similar direction, M. Spitzweck obtains in \cite{DFGTMSpit}, that the
triangulated category $\DMT_{/\F, \Z}^{gm}(X)$ of mixed Tate motives over $X$ (that
is before applying a $t$-structure 
and obtaining $\MTM (-)$) is the category $\Perf(B^{\bullet}_X)$ of perfect
representations of an affine derived group scheme over $\Z$ provided that $X$ is
a smooth connected $\F$-scheme 
of finite type ($\F$ is  any field) satisfying a weaker Beilinson-Soulé's
vanishing property \eqref{wBS} (see \cite[Theorem 2.2]{DFGTMSpit}).
Corollary 8.4 in \cite{SpitCP1S}  extends this construction to the case where
$X$ is simply smooth 
over $S$ and satisfying \eqref{wBS}. We give the necessary definition and
results below.

M. Spitzweck showed at corollaries 7.19 and 7.20 in \cite{SpitCP1S} that his
construction retrieves motivic cohomology.
\begin{prop}[{\cite[7.19 and 7.20]{SpitCP1S}}]\label{motcoh}For a smooth scheme
  $X$ over $S=\Spec(D)$ the spectrum of a Dedekind domain of mixed
  characteristic, one has  
\begin{align*}
%\simeq
 \Hom_{\SH(S)}(\Ssp[]{X}{\infty}), \MZ
S(p)[k]) & \simeq \Hom_{\DM(S)}(M_S(X), \Z_S(p)[k])  \\
& \simeq \Hom_{\DM(X)}(\Z_X(0), \Z_X(p)[k])  \\
 & \simeq \HH^k_{mot} (X, p)
\end{align*}

where $\HH^k_{mot} (X, p)$ denotes the motivic cohomology in the sense of Levine
\cite{LevMM};  that is it gives back the
higher Chow groups (see \cite{BlochACHKT, BKMTM, LevBHCG, LEVTMFG}) of $X$:
\begin{equation}\label{HmotHCG}
\HH^k_{mot}(X,p)=\CH^p(X,2p-k)
\end{equation}
\end{prop}

\begin{defn}[Beilinson-Soulé vanishing property]\label{def:BS}
Let $X$ be a smooth scheme over $S$. One says that $X$ satisfies
  Beilinson-Soulé vanishing property \eqref{BS} if and only if 
\begin{equation}\label{BS}
\Hom_{\SH(S)}(\Ssp[]{X}{\infty}), \MZ S(p)[k])=0, 
%\qquad \mbox{for all } p\geqs 0
%\mbox{ and }k \leqs 0 
\tag{BS}
\end{equation}
for all $p\geqs 0$ and $k<0$ and for all $p>0$ and $k=0$. 
\end{defn}
M. Spitzweck often needs only a weaker form of this property.
\begin{defn}[weak Beilinson-Soulé vanishing property]\label{def:wBS}
Let $X$ be a smooth scheme over $S$. One says that $X$ satisfies
  Beilinson-Soulé vanishing property \eqref{wBS} if and only if 
\begin{multline}\label{wBS}
\forall p>0,\quad 
\exists N \in \Z, \mx{ such that }
\mbox{ and }k < N ,\\
\Hom_{\SH(S)}(\Ssp[]{X}{\infty}), \MZ S(p)[k])=0.% \qquad  
\tag{wBS}
\end{multline}

\end{defn}
\begin{rem}\label{rem:Himot}
Note also as a consequence of \cite[in particular Theorem 7.10]{SpitCP1S}, one
 has also for $X$ a smooth irreducible $S$-scheme ($S$ regular)
\[
\Hom_{\DM(S)}(M_S(X),\Z(p)[k])=\left\{ 
\begin{array}{ll}
0 & \mbox{for  }p<0 \\
0 & \mbox{for }p=0 \mbox{ and }k\neq 0 \\
\Z & \mbox{for }p=0 \mbox{ and }k= 0 \\
\Of[X](X)^* & \mbox{for }p=1 \mbox{ and }k=1 \\
  \Pic(X) & \mbox{for }p=1 \mbox{ and }k=2 \\
\end{array}
\right.
\]
\end{rem}

\begin{thm}\label{BSZ}
$S=\Spec(\Z)$ satisfies \eqref{BS} property.
\end{thm}
\begin{proof}
The work of Borel \cite{BorCAG} and Beilinson \cite{BeiHRVLF} and the comparison
between groups of $K$-theory 
and motivic cohomology through higher Chow groups, shows that with $\Q$
coefficients $\Spec(\Q)$ satisfies the \eqref{BS} property. The difference
between $\Spec(\Z)$ and $\Spec{\Q}$ is concentrated in degree $1$ weight $1$
where the latter has an extra generator for each prime.  Thus $\Spec(\Z)$
satisfies the \eqref{BS} property with $\Q$ coefficient. As reviewed in
\cite[lemma 24]{KahnAKTACAG}, the \eqref{BS} property with
$\Z$-coefficient is a consequence of the Beilinson-Soulé vanishing
property with $\Q$ coefficients together with the Beilinson-Lichtenbaum
conjecture \cite[Conjecture 17]{KahnAKTACAG} which is equivalent to the
Bloch-Kato conjecture (see. \cite{SusVoeBKCMC}). Thanks to the work of V. Voevodsky this
last conjecture is a now a theorem proved in  \cite{VoeMCZ-lC}. This concludes
our theorem.
\end{proof}
\begin{rem}\label{BSnumberfield} Let $\F$ be a number filed and 
$\mc P$ a set of finite place of $\F$. Note that similar arguments show that
  the Beilinson-Soulé vanishing property also holds when $S$ is the spectrum of
  the ring $\mc O_{\F, \mc P}$ of $\mc P$-integers of $\F$;  that is when 
\[
S=\Spec(\mc O_{\F, \mc P}).
\]
\end{rem}

\section{Geometry of moduli spaces $\mob[n]$}\label{sec:m0n}
Let $n$ be an integer greater or equal to three and $\mo[n]$ be the moduli
space of curves of genus $0$ with $n$ marked points  over $\Sp(\Z)$ and
$\mob[n]$ its Deligne-Mumford compactification \cite{DMISCGG,PMSCKnud}. The
integer $l=n-3$ is the dimension of $\m_{0,n}$ and the boundary $\partial
\mob[n]=\mob[n]\sm \mo[n] $ is a strict normal crossing divisor whose
irreducible components are isomorphic to $\mob[n_1]\times \mob[n_2]$ with
$n_1+n_2=n+2$. For $\F$ a number field, we shall write ${\mo[n]}_{/\F}$ and
${\mob[n]}_{/\F}$ when the moduli spaces are considered over $\Sp{\F}$ and when the
context needs to be precised. If $S$ is a finite set with $n$ elements, we will
write $\mo[S]$ and $\mob[S]$ for $\mo[n]$ and $\mob[n]$. Note that if $(\p^1)^{|S|}_*$
denotes the set of all $n$-tuples of distinct points $z_s \in \p^1$, for $s\in S$, then
\[
\mo[S]= \PSL_2 \backslash (\p^1)^{|S|}_*,
\]
where $\PSL_2$ is the algebraic group of automorphisms of $\p^1$ and acts by
Möbius transformations.

Let $S= \{1, \ldots, n\}$. Recall that for any subset $S'$ of $S$, there exists a
natural map 
\[
f_{S'} : \mo[S] \lra \mo[S']
\] 
obtained by forgetting the marked points of $S$ which do not lie in $S'$. This
maps extends to a proper morphism
\begin{equation}\label{forgetbar}
\ol{f_{S'}} = \mob[S] \lra \mob[S'].
\end{equation}
\subsection{On the boundary of $\mob[n]$}\label{subsec:boundary}
Let $D$ be a codimension $1$ irreducible component of $\partial \mob[n]$ and
$D_0$ its open strata:
\[
D_0=D\sm (\bigcup_{D' \neq D} D\cap D')
\] 
where the unions runs through the codimension $1$ irreducible components of
$\partial \mob[n]$ different from $D$.
The union 
\[ 
\mo[n]\cup D_0=\mob[n]\sm (\partial \mob[n] \sm D_0) 
\]
is denoted by $\moD[n]{D}$ and the normal bundle of $D_0$ in $\moD[n]{D}$ is denoted
by $N_{D_0}$. The goal of this section is to prove that $N_{D_0}$ is trivial.

Let $S$ denote the set $\{1, 
\ldots n\}$. The moduli space $\mob[S]$ admits a stratification (cf. \cite{}). The
open strata is simply $\mo[S]$. A point in a codimension $k$ strata
represents a stable curve with $n$ marked points which  is a tree of $\p^1$'s
(because of the genus $0$) with the $n$ marked points spread on the
branches such that each $\p^1$, that is each branch, has at least three special
points (marked points and intersection points). Moving into a strata makes the
marked points move into their branch 
but they can not move from one branch to another. Following Keel's work \cite{IMSCKeel}, an
 irreducible component $D$ of $\partial \mo[S]$ is
given by its open strata. A point in this open strata represents a tree of $\p^1$
with two branches, the marked points being spread on these two branches. Thus
$D$ gives a two partition of $S$ which determines the open stratum of $D$. $D$
is then given by a subset $T_D$ of $S$ such that $T_D$ and and its complements
$T_D^C$ have at least $2$ elements and $D$ is isomorphic to
$\mob[T\cup\{e\}]\times \mob[T^c\cup \{e\}]$ with $e$ not in $S$. 

We fix three elements $i_0,i_1,i_2$ in $S$. The
correspondence between codimension $1$ irreducible components of
$\partial\mob[n]$  and partition $J \sqcup J^c$ of $S$ is now made $1$-to-$1$
by imposing that  $|J \cap \{i_0,i_1,i_2\}|\leqs 1$. For such a $J$, the
corresponding component of $\partial \mob[n]$ is denoted by $D^{J}$. 

We shall use the following notations. When the ambient space should be pointed
out, we will write $D_S^J$ instead of $D^J$. The open stratum of $D^J$ is 
\[
D_{0, S}^J=D_{S}^J \sm (\cup_{D' \neq D_{S}^J} D_{S}^J\cap D')
\] 
where the union runs through the codimension $1$ irreducible components $D'$ of
$\partial mob[S]$ different from $D_S^J$. Following the above notation, the
union $\mo[S]\cup D_{0,S}^J$ is denoted by $\mob[S]^{D^J}$. Note that
\[
\mob[S]^{D^J}=\mob[S]\sm (\cup_{D' \neq D_{S}^J} D_{S}^J\cap D')
\] 
where the union runs as above.

We suppose that $n \geqs 5$, thus we can assume that $T_D^c$ has
at least $3$ elements. Let $I=\{i_0, i_1,i_2,i_3\}$ be a subset of $S$ such that
$i_0$ is in $T$ and  $i_1$, $i_2$ and $i_3$ are elements of $T^c$. Let $S_0$ be
$S\sm \{i_3\}$. 
We consider the morphism
\[
\pi_{S_0\times I} : \mob[S] \xrightarrow{\ol{f_{S_0}}\times \ol{f_I} }
\mob[S_0]\times \mob[I]. 
\]

\begin{lem}\label{pi(D0)}
Let $S=\{1,\ldots , n\}$, $T\subset S$, $I={i_0,i_1,i_2,i_3}$ and $S_0$ be as
above.  Then the image of $D^T_0$ by $\pi_{S_0\times I}$ satisfies:
\[
\pi_{S_0\times I}(D^T_0) \subset D^T_{0, S_0} \times \mo[4].
\]   
\end{lem}
\begin{proof}
Let $P$ be a point in $D^T_{0}$. As $P$ is in the open strata, $P$ represents a
tree of $\p^1$ having only two branches with $n$ marked points spread on each
branch accordingly to the partition $T\sqcup T^c$ of $S$. The forgetful
morphisms make at worst the number of branches decrease. Hence $\ol{f_{S_0}}(P)$
  have at most $2$ branches and thus is at worst in the open strata of a
  codimension $1$ irreducible component of $\mob[S_0]$. In one hand, note that the as
  $|T|\geqs 2$ and $i_3\notin T$, 
  $\ol{f_{S_0}}(P)$ is in $D^T_{S_0}$ and thus is in $D_{0,S_0}^T$. In the other hand $T \cap
    \{i_0,i_1,i_2,i_3\}=\{i_0\}$ and the tree of $\p^1$ corresponding to $P$ can
  not remain stable under $\ol{f_I}$, thus $\ol{f_{I}}(P)$ represents a single
  $\p^1$ with $n$ marked points and thus is in  $\mo[4]$.
\end{proof}

\begin{prop}\label{NDtrivial}
Let $n$ be greater or equal to $4$ and let $D$ be a codimension $1$ irreducible
component of $\partial \mob[n]$. Then, the normal bundle $N_{D_0}$ is trivial.
\end{prop}
\begin{proof}
The proof proceed by induction on $n$ and is clear for $n=4$. 

Assuming that $n \geqs 5$, we write as above $S=\{1, \ldots n n\}$ and $D$ as
$D^T$ for some $T\subset S$. The cardinal of $S$ being at least $5$, we can
assume that $|T| \geqs 2$ and that $|T^c|\geqs 3$. In order to have a $1$-to-$1$
correspondence between codimension $1$ irreducible component of $\partial
\mob[S]$ and $S$ partition as described above, we chose $i_0$ in $T$ and $i_1$
and  $i_2$ in $T^c$. $T^c$ having at least three elements, we chose there a
third elements $i_3$ different from $i_1$  and $i_2$.  

 Keel's work \cite[Lemma 1]{IMSCKeel} shows that the morphism
\[
\pi_{S_0\times I} : \mob[S] \xrightarrow{\ol{f_{S_0}}\times \ol{f_I} }
\mob[S_0]\times \mob[I] 
\]
is given by a succession of blow-ups along regular, smooth, codimension
$2$ subschemes and whose exceptional divisors are codimension $1$ irreducible
components of $\partial \mob[n]$ of the form :
\[
D^{J\cup \{3\}} \qquad \mx{with } \left\{ 
\begin{array}{l} J \subset S_0, \quad |J| \geqs 2\\
  \mx{and } |J\cap \{i_0,i_1,i_2\}|\leqs 1.
\end{array} \right. 
\]

In particular, $\pi_{S_0\times I}$ is an isomorphism outside the
exceptional divisor. Hence, the image of $\moD[S]{D^T}$ by  $\pi_{S_0\times I}$
is an open of  
$\mob[S_0] \times \mob[4]$. 

As $T$ is also a subset of $S_0$, let  $D^T_{S_0}$ be the corresponding
codimension $1$ component of $\partial \mob[S_0]$ and $D_{0, S_0}^T$ its open
stratum. Lemma \ref{pi(D0)} above shows that
\[
\pi_{S_0\times I}(D^T_0) \subset D_{0,S_0}^T \times \mo[4].
\]

Thus, the image of $\pi_{S_0\times I}(\moD[S]{D^T})$ is open in $\mob[S_0]\times
\mob[4]$ and is included in
$\moD[S_0]{D^T}\times \mob[4]$ which is also open in $\mob[S_0]\times
\mob[4]$. As a consequence, $\pi_{S_0\times I}(\moD[S]{D^T})$ is open in
$\moD[S_0]{D^T}\times \mob[4]$. 
 
The morphism $\pi_{S_0\times I}$ being an isomorphism away from the exceptional
divisor, the triviality of $N_{D^T_0}$ in $\moD[S]{D^T}$ is  equivalent to the
triviality of the normal bundle of $\pi_{S_0\times I}(D^T_0)$ in $\pi_{S_0\times
  I}(\moD[S]{D^T})$ which by the above discussion is a consequence of the
triviality of $N_{D^T_{0,S_0}}$ in $\moD[S_0]{D^T_{0,S_0}}$. The Proposition follows by
induction, the case $\mob[4]\simeq \p^1 $ with $\partial \mob[4]\simeq \{0, 1,
\infty\}$ being trivial.

\end{proof} 
\subsection{The motives of $\mob[n]$}
Now, $S=\Spec(\Z)$. The main goals of this subsection is to prove that the
motive $M_S(\mob[n])$ is a (finite) direct sum of motives of the type
$\Z_S(p)[2p]$ with $p\geqs 0$ and satisfy \eqref{BS} property (see Theorem
\ref{MobnTBS}).

The key points are the decomposition of $M_S(\mob[n])$ into Tate motives and the
Beilinson-Soulé property for the base scheme $S=\Spec(\Z)$.

\begin{defn} Let $X$ be a smooth scheme over $S$. We say that $X$ is
  \emph{effective of type Tate type $(p,2p$)} or simply of type $ET$ when
  $M_S(X)$ is a finite direct  sum of motives $\Z_S(p_i)[2p_i]$ with $p\geqs 0$.
\end{defn}

A direct application of the blow-up formula (Proposition \ref{Blup}) gives the
following.
\begin{lem}\label{lem:blET}Let $X$ be a smooth scheme over $S$ and $Z$ a smooth
  closed subscheme 
  of $X$. We assume that both $X$ and $Z$ are of type $ET$. Then the blow-up
  $\Bl_Z(X)$ of $X$ with center $Z$ is also of type $ET$. 
\end{lem}

\begin{thm}\label{MobnTBS} Let $n$ be an integer greater or equal to $3$. The motive
  $M_S(\mob[n])$ is isomorphic to
\[
M_S(\mob[n])=\oplus_i \Z_S(p_i)[2p_i]
\]
where the above sum is finite and the $p_i$ are positive (or zero). Moreover
$M_S(\mob[n])$ satisfy \eqref{BS}; that 
is for any integers $p$ and $k$ such that $p\geqs 0$ and $k< 0$, or $p>0$ and
$k=0$ one has  
\[ 
\Hom_{\SH(S)}(\Ssp[]{\mob[n]}{\infty}), \MZ S(p)[k])
=\Hom_{\DM (S) }(M_S(\mob[n]), \Z_S(p)[k])=0
\]
\end{thm} 

\begin{proof}
Note that the second part of the theorem follows directly from the first part
using Lemma \ref{BSZpk} below.

The proof of the first part is by induction on $n$. 

Note that $\mob[3]$ is  isomorphic to  $S=\Spec(\Z)$ and
$\mob[4]$ is simply $\p^1_S$. Hence, using the \eqref{BS} property for
$S=\Spec(\Z)$ (Lemma \ref{BSZ}) and the projective bundle formula, $\mob[3]$ and
$\mob[4]$ are of type $ET$.

Fix $n \geqs 5$. Let $I_n$ denotes the set $\{1, \ldots, n\}$ (denoted by $S$ in
the previous section). Keel in \cite[Theorem 1 and 2]{IMSCKeel} proves that the morphism
\[
\mob[I_n] \lra \mob[I_{n-1}]\times \mob[I_4]
\]
is a sequence of blow-ups 
\[
\mob[I_n] \st{\sim}{\ra} B_{n-3} \ra \cdots \ra B_k 
\ra \cdots  \ra B_1 =  \mob[I_{n-1}]\times \mob[I_4]
\]
where $B_{k+1}\lra B_k$ is the blow-up along disjoints centers isomorphic to
some irreducible components of $\partial \mob[I_{n-1}]$.

As $B_1 \simeq \mob[I_{n-1}]\times \mob[I_4]$, the induction hypothesis and
Künneth formula show that $B_1$ is of type $ET$. Note that irreducible
components of $\mob[I_{n-1}]$ are isomorphic to $\mob[n_1]\times \mob[n_2]$ with
$n_1+n_2=n+1$. Thus Künneth formula and the induction hypothesis show that the
centers of the blow-up $B_{k+1}\ra B_{k}$ are also of type $ET$.

 Now, using a proof similar to the proof of Proposition 4.4 in \cite{SouMDS}, an
 induction on $k$ and the blow-up formula prove that $B_k$ is of type 
 $ET$ for all $k$. Hence $\mob[I_n]\simeq B_{n-3}$ is also of type $ET$. 
\end{proof}
Note that \cite{SouMDS} use a cohomological setting which explains the minus
signs in shift and twist. One could by-pass part of Keel's result in
\cite{IMSCKeel} by remarking that the map
\[
\mob[I_n]\lra \prod_{i=4}^n\mob[\{1,2,3,i\}]
\] 
crashes down all irreducible components of the form $D^T$ with $|T\cap
\{1,2,3\}|\leqs 1$ and $|T| \geqs 3$. Normalizing the marked point $z_1$, $z_2$
and $z_3$ to $1$ $\infty$ and $0$ respectively, one obtains that  $\mob[I_n]$ is
the results of blowing up $(\p^1)^{n-3}$ along the poset given by the
all the intersections of the divisors $t_i=t_j$ and $t_i=\ve$ with $i\neq j$ and
$\ve=0,1,\infty$. This is exactly the situation of \cite[Proposition
4.4]{SouMDS} and its proof has to be modified each times it uses the blow-up
formula in order to take into account the $ET$ type property.
 
\begin{lem}\label{BSZpk} Let $X$ be a smooth scheme over $S$. Assume that $X$ is
  effective of 
  type Tate $(p,2p)$, then $M_S(X)$ satisfy the Beilinson-Soulé property \eqref{BS}.
\end{lem}
\begin{proof}
By definition, $M_S(X)$ is a direct sum of Tate motives of the form
$\Z_S(i)[2i]$ for $i\geqs 0$. Using Proposition \ref{motcoh}, in order to show
that $M_S(X)$ satisfies 
\eqref{BS},  it then is enough to show that,  for any $i\geqs 0$ and for
any $p\geqs 0$ and 
any $k< 0$  or any $p>0$ and $k=0$, one has
\[
\Hom_{\DM(S)}(\Z_S(i)[2i],\Z(p)[k])=0.
\]
However the above $\Hom$ group is simply
\[
\Hom_{\DM(S)}(\Z_S(0), \Z(p-i)[k-2i]).
\]
If $p-i$ is less or equal to $0$, one can use Remark
\ref{rem:Himot}. When
$p-i>0$,  then the results follows from the \eqref{BS} property of $S=\Spec(\Z)$
given at Theorem \ref{BSZ} because $k-2i<0$. 
\end{proof}

\begin{coro}\label{D-M0n-BS} Let $n \leqs 4$ and $D$ an irreducible component of
  $\partial \mob[n]$. Then, $D$ is of type $ET$ and satisfies the \eqref{BS}
  property.

Moreover, if $\mc S$  is  a non empty intersection of $k$ irreducible
codimension $1$ components of
  $\partial \mob[n]$, then $\mc S$ is of type $ET$ and satisfies the \eqref{BS}
  property.    
\end{coro}
\begin{proof}

The closed strata $\mc S$ is isomorphic to a product 
\[
\mob[l_1+3]\times \mob[l_2+3] \times \cdots \times  \mob[l_{k+1}+3]
\] 
with $l_1+l_2+\cdots +l_{k+1}=n-3-k$ (see \cite{PSEMC}). Then the corollary  is
a direct consequence 
of Künneth formula and Theorem \ref{MobnTBS}.
\end{proof}
\subsection{The motive of $\mo[n]$} The motives of the open moduli spaces of
curve $\m_{0,n}$ are mixed Tate and satisfy the \eqref{BS} property. This is proved in the
following section.

First we recall some facts about the boundary of $\mob[n]$ and its stratified
structure. It has already been remarked that $\partial \mob[n]=\mob[n]\sm
\mo[n]$ is a normal 
crossing divisor (cf. \cite[Theorem. 2.7]{PMSCKnud}). Let $\bar S$ be the
intersection of $k$ irreducible codimension $1$ components of $\partial
\mob[n]$.  
%(which is
%also irreducible in the genus $0$ case).  
Then $\bar S$ is
isomorphic to the product of $k+1$ moduli spaces of curves
\[
\bar S \simeq \mob[n_1] \times \cdots \times \mob[n_{k+1}]
\]
such that $\sum_{i=1}^{k+1} (n_i-3)=n-3-k$. 
%\cite{??}\marginote{ref ? \\
%  Harris/Moduli ?}. 

Writing $\partial \mob[n]$ as the union of its irreducible components 
\[
\partial \mob[n]=\bigcup_{i=1}^{N} D_i,
\]
one can assume that $\bar S=\cup_{i=1}^k D_i$.
The open strata $\mathring S$ is defined as 
\[
\mathring S= \bar S \sm \left(S\cap\left( \bigcup_{i=k+1}^N D_i \right) \right)
\]
and is isomorphic to 
\[\
\bar S \simeq \mo[n_1] \times \cdots \times \mo[n_{k+1}].
\]
\begin{thm}\label{thm:M0nMTMBS}
Let $n$ be an integer greater or equal to $3$. The motive
  $M_S(\mo[n])$ is in $\DMT_{/S, \Z}(S)$ the triangulated category of mixed Tate
  motives. Moreover the motive 
$M_S(\mo[n])$ satisfies \eqref{BS}; that 
is  one has 
\[ 
\Hom_{\SH(S)}(\Ssp[]{\mo[n]}{\infty}), \MZ S(p)[k])
=\Hom_{\DM (S) }(M_S(\mo[n]), \Z_S(p)[k])=0
\]
for all $p\geqs 0$ and $k<0$ and for all $p>0$ and $k=0$. 
\end{thm}

This statement holds in a more general situation. Let $X_0$ be a smooth
projective scheme
over $S$ whose motive $M_S(X_0)$ is in   $ \DMT_{/S, \Z}(S)$ and satisfies
\eqref{BS}. Let
$Z_0=\cup_{i=1}^l Z_i$ a strict normal crossing divisor of $X_0$. Assume that any
irreducible components of any intersection of the $Z_i$'s has a motives in  $
\DMT_{/S, \Z}(S)$ and satisfies \eqref{BS}. Let $U=X_0\sm 
Z_0$. 
\begin{thm} \label{thm:NCD-BSprop}
$M_S(U)$ is in $ \DMT_{/S, \Z}(S)$ and satisfies \eqref{BS}
\end{thm} 
\begin{proof} The proof is a double induction on the dimension $n$ of $X_0$ and $l$.

Let $Z'=\cup_{i=1}^{l-1}Z_i$, $X=X_0\sm Z'$. The intersection  $Z=Z_l\cap X$ is
of codimension $d=1$.  
The Gysin triangle \eqref{eq:GysinT} insures that $M_S(U)$ sits in the
distinguished triangle
\[
\lra M_S(U) \lra M_S(X) \lra M_S(Z)(d)[2d] \lra M_S(U)[1]\lra \cdots
\]
 
 Applying the $\Hom_{\DM(S)}$ functor, one gets and exact sequence
\begin{multline*}
\HH^k_{mot}(X, p) \lra \HH^k_{mot}(U,p)=\Hom_{\DM(S)}(M(U)[1],\Z_S(p)[k+1]) \\ 
\lra \Hom_{\DM(S)}(M_S(Z)(d)[2d], \Z_S(p)[k+1])=\HH^{k+1-2d}_{mot}(Z,p-n). 
\end{multline*}

When $n=1$ and $l=1$ $X=X_0$ and $Z=Z_1=Z_0$. Hence both are in
$ \DMT_{/S, \Z}(S)$ and satisfy \eqref{BS} which implies the theorem for $U$.  

When $n>1$ or $k>1$, by induction $X$ is in $ \DMT_{/S, \Z}(S)$ and satisfies
\eqref{BS}.  $Z=Z_l\cap X$ is equal to 
\[
Z=Z_l\sm \left( \bigcup_{i=1}^{l-1}Z_i \cap Z_l\right).
\] 
By induction, $Z$  is in $ \DMT_{/S, \Z}(S)$ and satisfies \eqref{BS}.
Thus, the above  exact sequence, induced by the Gysin triangle, implies the theorem for
$U$.
\end{proof}

\begin{proof}[Proof of Theorem \ref{thm:M0nMTMBS}] We apply Theorem
  \ref{thm:NCD-BSprop} to the case where  $X=\mob[n]$ and $Z_0=\partial
  \mob[n]$. In this case, Theorem \ref{MobnTBS} and Corollary
\ref{D-M0n-BS} insure that the hypothesis are satisfied. Note that is this case,
$Z$ is an open codimension $1$  stratum of the compactification, hence it is
isomorphic to a 
product of open moduli spaces of curves. One could do the induction only for the
moduli spaces of curves case.
\end{proof}
%Note that in this
%particular case   

\begin{rem}
\begin{itemize}
\item Note that any strict normal crossing divisor $Z_0$ of $X_0$ induces a
  stratification of $X_0$ where the strata are given by irreducible components
  of the intersection of the $Z_i$'s. The above theorem remains valid under the
  same hypothesis when $U=X\sm (\cup_{i\in I} \bar{\mathscr  S_i})$ is the complement
  of a union of closed strata defined by the divisor $Z_0$; where the strata
  $\bar{ \mathscr S_i}$ have maximal dimension $d_i$ and $I$ is minimal a description
  of $U$. This remove some ambiguities in the choices of the strata. In this
  case, the proof goes by induction on the dimension of $X_0$, $d=\max(d_i)$ and
  the number $k$ of strata of dimension $d$. As above it follows from the Gysin
  triangle and the long exact sequence for the $\HH^k_{mot}$. Note that closed
  strata of dimension $0$ are disjoint and that open strata (i.e. closed strata minus
  closed strata of lower dimension)  of dimension $d$ are disjoint.
\item Duality and  Gysin morphism as explained in Section \ref{sec:spitweckmot}, and
  given by F. Déglise works \cite{AGTIIDeg},  gives relative motivic cohomology
  $\HH^k_{mot}(X \sm A ; B)$ where $X$ is 
  smooth projective, and $A$ and $B$ are two strict normal crossing divisors
  sharing no common irreducible components. This is explained by M. Levine in
  \cite{LevMM}.
\item The theorem also shows that for any dihedral structure $\delta$ on
  $\{1,\ldots, n\}$ 
   the motive $M_S(\mo[n]^{\delta})$
  of $\mo[n]^{\delta}$ (which was defined by F. Brown in  \cite[Section
  2.2]{BrownMZVPMS}) is also mixed Tate; that is a motive in $ \DMT_{/S,
    \Z}(S)$; and satisfies the \eqref{BS} property.
\end{itemize}
\end{rem}

\section{A motivic Grothendieck-Teichmüller group}
\label{sec:GTmotgen}

\newcommand{\Gb}{G^{\bullet}}
\newcommand{\Kb}{K^{\bullet}}
This section defines a ``derived'' integral motivic
Grothendieck-Teichmüller group over $\Z$: $\GT_{\Z}^{mot}(\Z)$. Here ``derived''
refers to the fact 
that the construction is based on the triangulated category $ \DMT_{/S,
  \Z}(S)$.  
M. Spitzweck's work gives for any $n \geqs 3$ an equivalence between $ \DMT_{/S,
  \Z}(S)$ and the perfect representations of a ``derived group'' $\Gb(\mo[n])$.  These
groups sit as middle terms in shorts exact sequences relating  $\Gb(\mo[n])$,
$\Gb(\mo[3])=\Gb(\Spec(\Z))$ and a ``geometric part'' $\Kb_n$. These exact sequences are
compatible with the natural morphisms in the tower of the $\mo[n]$ (forgetting
marked points and ``embedding of codimension $1$ component'').

$\GT_{\Z}^{mot}(\Z)$ is then defined as the automorphism of the tower given by
the $\Kb_n$. A non derived version will be presented in the next section using
rational coefficients where the $t$-structure is available. In this rational
context and working over the spectrum of a number field, 
M. Levine's work \cite{LEVTMFG} 
identifies the non derived part
of $\Kb_n$ with
Deligne-Goncharov motivic fundamental group $\pi_1^{mot}(\m_{0,n})$
\cite{DG}. Hence the description of $\Kb_n$ as a ``geometric part''.
\subsection{Tangential based points and normal bundle}
This section describes the last requirement for developing a motivic
Grothendieck-Teichmüller construction:
\begin{itemize}
\item A natural functor $\DMT_{/S, \Z}(X\sm Z) \lra \DMT_{/S, \Z}(N^0_Z) \lra
  \DMT_{\S, \Z}(Z)$ where $N_Z^0$ denotes the normal bundle of $Z$ in $X$ minus
  its zero section. This functor allows us to obtain a derived group morphism 
$G(\mo[k])\times G(\mo[n])\lra G(\mo[n])$ induced by the inclusion of an
irreducible component $D\simeq \mob[k]\times \mob[l]$ of $\mob[n]$ into
$\mob[n]$. This is a motivic version of the morphisms between fundamental groups
 presented in \cite{PSEMC} in the topological context  or in
 \cite{NakamuraCUMRGT} in the etale case. 
\item Motivic \emph{tangential base points} or \emph{motivic base points at
    infinity}. In general based points provide an augmentation to differential graded
  ($E_{\infty}$) algebras  underlying the description of mixed Tate categories
  as comodule categories, more specifically when one deals with the relative
  situation \cite{LEVTMFG}. They also  give sections in the (derived) group setting
  to the morphism  $p^*: \DMT_{/S, \Z}(S) \lra \DMT_{/S,\Z}(X)$ induced by the
  structural morphism $p: X \ra S$. Tangential based points are used to
  compensate the lack of $S$-points (such as in the case of $\ps$ over
  $\Spec(\Z)$) and to preserve symmetries.
\end{itemize}

Let $n$ be an integer greater or equal to $4$.
Let $D$ be an irreducible component of $\partial \mob[n]$ and $D_0$ its open
strata (cf. Section \ref{subsec:boundary}). The open strata is isomorphic to 
 \[
D_0\sim \mo[n_1]\times \mo[n_2]
\]
with $n_1+n_2=n+2$. 

\begin{prop}\label{divisorfunctor}
There is a natural functor 
\[
 \mc L_{D,\mo[n]}^{DM} : \DMZ{\mo[n]} \lra \DMZ{D_0}
\]
sending Tate object to Tate object 
and hence inducing a natural
functor
\[
 \mc L_{D,\mo[n]} : \DMT_{/S,\Z}(\mo[n]) \lra \DMT_{/S,\Z}(D_0)
\]

Moreover its composition with the ``structural functor'' $p^* : \DMT_{/S,\Z}(S)
\lra \DMT_{/S,\Z}(\mo[n])$ is isomorphic to $p_{D^0}^*:\DMT_{/S,\Z}(S)\lra \DMT_{/S,\Z}(D_0)$.
\end{prop}
\begin{proof}
Let $X$ a be a smooth scheme over $S$, $Z\st{i}{\lra} X$ a regular closed
embedding such that $Z$ is  smooth over $S$. In our case we can also assume that
$Z$ is a divisor of $X$. Let $X^{0}$ be the open complement and
$N_Z^{0}$ the normal bundle of $Z$ in $X$ with zero section removed. 
In \cite{OAMSpit}, M. Spitzweck  defined, as a consequence of his 
Proposition 15.19, a natural functor
\[
\mc L_{X,Z}^{DM} : \DMZ{X^0} \lra \DMZ{N^0_Z}.
\] 

Then, we will apply this functor to the situation $X=\mo[n]\cup D_0$ and compose it with
the pull-back functor induced by an everywhere non-zero 
$S$-section $\sigma : D_0 \lra N_{D_0}^0$ given by Proposition
\ref{NDtrivial}. In order to show that it sends Tate objects to Tate objects, we
need to enter a little in the construction of the functor $\mc L_{X,Z}$ which
relies for its geometric part on the (affine) deformation to the normal cone. 

The deformation to the normal cone is a key geometrical construction need  in
the studies of Gysin maps and was explicitly used and formalized by W. Fulton
\cite{FultonIT}. Later on, it plays an important role in defining
\emph{specialization maps} for example in microlocal theory of sheave
\cite{KaScha94}, and was developed and generalized in \cite{RostCGC} and
\cite{IvorraCMIAA} to higher deformations in order to
study the $A_{\infty}$ structure of cycle modules. 

The starting situation is the following :
\begin{equation}\label{XZdiag}
\begin{tikzpicture}
\matrix (m) [matrix of math nodes,
 row sep=3em, column sep=2em,%]%, 
 text height=2.0ex,  text depth=0.25ex] 
{ Z & X & X^0    \\
 & S & \\
};
\path[->,font=\scriptsize]
(m-1-1) edge node[above, math mode] {i_Z}  (m-1-2)
(m-1-3) edge node[above, math mode] {j_{X^0}}  (m-1-2)
(m-1-1) edge node[left, math mode] {p_Z}  (m-2-2)
(m-1-2) edge node[auto, math mode] {p_X} (m-2-2)
(m-1-3) edge node[right, math mode] {p_{X^0}} (m-2-2)
;
\node[xshift=2.8em,yshift=-0.25ex] at (m-1-3.center) {$ =X \sm Z^{}$};
%\path[->,font=\scriptsize, bend left]
%(m-1-2) edge node[auto] {$\otimes^L\MZ{S}$} (m-1-1);
\end{tikzpicture}
\end{equation}
Then, one considers the blow-up $\Bl_{Z \times \{0\}}X\times \A^1_{S}$ of $X \times
\A^1_S$ along  $Z \times \{0\}$:
\[
\pi : \Bl_{Z\times\{0\}}X \times \A^1_S \lra X \times \A^1 .
\]
Its restriction to $X \times {\Gm}_S$ is by definition isomorphic to $X \times
\Gm_S$ while its restriction to $X\times \{0\}$ has two components, one of them
is $\Bl_Z X$ and the other one is $\p(N_Z \oplus \mc O_Z)$ where $N_Z$ denotes the
normal bundle of $Z$ in $X$. They intersect each
other on $\p(N_Z)$ which is the exceptional divisor of $\Bl_ZX$. The deformation
of $Z$ to the normal cone is defined as 
\[
D(X,Z)=\left(\Bl_{Z\times \{0\}}X\times \A^1_S\right) \sm \Bl_ZX.
\]
In terms of spectrum, if $\mc J_Z$ denotes the sheaf of ideal defining $Z$, the
deformation $D(X,Z)$ is given by $\Spec(\mc A_{X,Z})$ where  
\[
\mc A_{X,Z}=\oplus_{n \in \Z} \mc J_Z^nt^{-n} \subset \mc O_X[t, t^{-1}]
\]
with the convention that $\mc J_Z^n=\mc O_X$ as soon as $n\leqs 0$. Dropping the
subscript $S$ when the situation is clear enough, the geometric situation is
described by the following diagram
\begin{equation}\label{DXZdiag}
\begin{tikzpicture}
\matrix (m) [matrix of math nodes,
 row sep=4em, column sep=1.5em,%]%, 
 text height=2.0ex,  text depth=0.25ex] 
{Z & N_Z & & D(X,Z) & & X\times \Gm & X   \\
Z\times\{0\}& &\hphantom{AA} & & X \times \A^1 & & X \times \Gm \\
&\{0\} & & \A^1 & & \Gm & \\
};
\path[->]
(m-1-2) edge node[above, mathsc] {i_{N_Z}}  (m-1-4)
(m-1-6) edge node[above, mathsc] {j_{X \times\Gm}}  (m-1-4)
(m-1-2) edge node[right, mathsc] {p_{N_Z}}  (m-2-1)
(m-1-4) edge node[right, mathsc] {\pi|_{D(X,Z)}} (m-2-5)
(m-1-6) edge node[left, mathsc] {=} (m-2-7)
(m-1-2) edge%[bend right=45] 
node[auto, mathsc] {f|_{\{0\}}}  (m-3-2)
(m-1-4) edge%[bend left=40] 
node[auto, mathsc] {f} (m-3-4)
(m-1-6) edge%[bend left=45] 
node[auto, mathsc] {f|_{\Gm}} (m-3-6)
(m-2-1) edge (m-3-2)
(m-2-5) edge (m-3-4)
(m-2-7) edge node[mathsc, auto]{p_{\ra \Gm}} (m-3-6)
(m-3-2) edge node[above, mathsc] {i_{0}}  (m-3-4)
(m-3-6) edge node[above, mathsc] {j_{\Gm}}  (m-3-4)
(m-1-1) edge node[left, mathsc] {\sim} (m-2-1)
(m-1-1) edge[bend left=45, dotted] node[above,  mathsc] {s_0} (m-1-2)
(m-1-2) edge node[auto, mathsc] {p_{N_Z}} (m-1-1)
(m-1-6) edge node[auto, mathsc] {p_{\ra X}} (m-1-7)
(m-2-7) edge node[right, mathsc] {p_{\ra X}} (m-1-7)
;
%\node at (m-1-3.east) {$ =X \sm Z$};
%\path[->,font=\scriptsize, bend left]
%(m-1-2) edge node[auto] {$\otimes^L\MZ{S}$} (m-1-1);
\end{tikzpicture}
\end{equation}
In the above diagram the two ``big rectangles'' are Cartesian. Note that the map
$f$ and hence 
maps $f|_{\{0\}}$ and $f|_{\Gm}$ are smooth because $Z$ itself is smooth over
$S$. This was remarked by J. Ayoub in \cite[Beginning of Section 1.6.1 after diagram
(1.37)]{Ayoub6OGI}.  

The open deforamtion $D^0(X,Z)$ is obtain by removing in $D(X,Z)$ the strict
transform  of $Z\times \A^1$ ; that is the closure in $D(X,Z)$ of $Z \times \Gm$.
The  properties of  $D^0(X,Z)$  are resumed in the
following Cartesian diagram
\begin{equation}
\label{D0XZdiag}
\begin{tikzpicture}
\matrix (m) [matrix of math nodes,
 row sep=4em, column sep=3em,%]%, 
 text height=2.0ex,  text depth=0.25ex] 
{ Z &  N_Z^0& D^0(X,Z)  &X^0\times \Gm & X^0   \\
 & \{0\} & \A^1 &  \Gm  \\
};
\path[->]
(m-1-2) edge node[auto, mathsc]{p^0_{N_Z}} (m-1-1)
(m-1-2) edge node[auto, mathsc]{f|_{\{0\}}}(m-2-2) 
(m-1-3) edge node[auto, mathsc]{f}(m-2-3) 
(m-1-4) edge node[auto, mathsc]{p_{\ra \Gm}}(m-2-4) 
(m-1-2) edge node[auto, mathsc]{i_{N_Z^0}}(m-1-3)
(m-1-4) edge node[auto, mathsc]{j_{X^0\times \Gm}}(m-1-3) 
(m-2-2) edge node[auto, mathsc]{i_{\{0\}}}(m-2-3)
(m-2-4) edge node[auto, mathsc]{j_{\Gm}} (m-2-3)
(m-1-4) edge node[auto, mathsc]{p_{\ra X^0}}(m-1-5)
;
\end{tikzpicture}
\end{equation}
which is ``an open immersion'' of the previous one with closed complement given
over $\{0\}$ and $\Gm$ by $s_{0}(Z) $ and $Z$ respectively.

From this geometric situation M. Spitzweck obtains \cite[Proposition
15.19]{OAMSpit} an isomorphism 
\begin{equation}\label{Spiti*j*p*}
i_{Z}^*j_{X^0\,*} \MZ{X_0} \simeq p^0_{N_Z\,*}\MZ{N^0_Z}
\end{equation}
by comparing the inclusions of $X \st{\sim}{\ra} X\times\{0\}$ and $X
\st{\sim}{\ra} \times \{1\}$ 
in the strict transform of $X \times \A^1$ in $D(X,Y)$ and similarly for $Z$. 

Then  M. Spitzweck uses one of his main results \cite[Corollary 15.14]{OAMSpit}
to identify the homotopy category of modules over $p^0_{N_Z\, *}\MZ{N_Z^0}$ with
 the full triangulated subcategory of $\DMZ{S}$ generated by homotopy colimits
 of pull-back  by $p^0_{N_Z}$ of objects from $\DMZ{Z}$.  The composition of
 $\DMZ{X^0} \lra \mc H(i_Z^*j_{X^0\, *}-Mod)$ with the two identifications gives
 a functor 
\[
\mc L_{X,Z}^{DM} : \DMZ{X^0} \lra \DMZ{N^0_Z}.
\]
 Let $p_{X^0} : X^0 \lra S$ be the structural morphism. The composition of
 $p_{X^0}^* : \DMZ{S} \lra \DMZ{X^0}$ with $\mc L_{X,Z}$ is 
 isomorphic to $\DMZ{S} \lra \DMZ{N_Z^0}$ induced by the structural morphism of
 $N_Z^0$ because 
\begin{itemize}
\item the compatibility with objects lifted from the base  is given at the end of
  Corollary 15.14 in \cite{OAMSpit} ;
\item the condition of corollary 15.14, asking that $M\otimes p^0_{N_Z \, *}
  \MZ{N_Z}$ is isomorphic  to $p^0_{N_Z\, *} p_{N_Z}^{0\, *}(M)$ for any $M$ in
  $\DMZ{Z}$, is satisfied. 
\end{itemize} 

Note that these two same properties insure that $\mc L_{X, Z}$ sends Tate
object to Tate object.

Note that this material has also been developed in \cite{SpitMALM} with some
more details.
\end{proof}

\begin{rem}\label{nearbyDivisor}
We develop in this remark another approach to limit motives : the nearby cycle
functor \cite{Ayoub6OGII}. This method has been used by J. Ayoub  in 
\cite{AyoubAHGFMccnII,  AyoubRCP} in the case of curve over field. The following
construction  agrees with  M. Spitzweck's one; see \cite{SpitMALM, OAMSpit}. We
explain below how deformation to the normal cone allows us to obtain a limit
motive functor from Ayoub's nearby cycle functor. For the rest of this remark,
we assume that  Ayoub's formalism is
available ; that is we assume that the functor $\DMZ[{/S,R}]{}$ associating to $X \in
\Sm[S]$ the triangulated category
$\DMZ[/S,R]{X}$ (using $R$ coefficients) is coming from a monoidal stable 
homotopic algebraic derivator on diagrams of quasi-projective scheme over $S$. 
This assumption applies directly to our situation when working with rational
coefficients ($R=\Q$) and the Beilinson's $E_{\infty}$-ring spectrum $\MZ[\Q]{X}$ as in
section \ref{rationalcoef} below ($S=\Spec(\Z)$ is omitted from the notation).

We give again  diagram \eqref{D0XZdiag} :
\[
\begin{tikzpicture}
\matrix (m) [matrix of math nodes,
 row sep=4em, column sep=3em,%]%, 
 text height=2.0ex,  text depth=0.25ex] 
{ N_Z^0& D^0(X,Z)  &X^0\times \Gm & X^0   \\
  \{0\} & \A^1_S &  \Gm_S  \\
& S & \\
};
\path[->]
(m-1-1) edge node[auto, mathsc]{f|_{\{0\}}}(m-2-1) 
(m-1-2) edge node[auto, mathsc]{f}(m-2-2) 
(m-1-3) edge node[auto, mathsc]{p_{\ra \Gm}}(m-2-3) 
(m-1-1) edge node[auto, mathsc]{i_{N_Z^0}}(m-1-2)
(m-1-3) edge node[auto, mathsc]{j_{X^0\times \Gm}}(m-1-2) 
(m-2-1) edge node[auto, mathsc]{i_{\{0\}}}(m-2-2)
(m-2-3) edge node[auto, mathsc]{j_{\Gm}} (m-2-2)
(m-1-3) edge node[auto, mathsc]{p_{\ra X^0}}(m-1-4)
(m-2-1) edge node[auto, mathsc] {=}(m-3-2)
(m-2-2) edge (m-3-2)
(m-2-3) edge node[auto, mathsc] {q_{\Gm}}(m-3-2)
;
\end{tikzpicture}
\]
which corresponds to the situation of a ``specialization functor'' over the base $\A^1$ 
as described by Ayoub \cite[Section 3.1 ; 3.4 and 3.5]{Ayoub6OGII}. The nearby
cycles functor from Ayoub gives us 
\[
\Psi_f : \DMZ[R]{X\times \Gm} \lra \DMZ[{R}]{N_Z^0}.
\]
The limit motives functor is then obtained by composing with $p_{\ra X^0}^*$:
\[
\mc L_{X,Z}^{DM\, \Psi} : \DMZ[{R}]{X^0}\st{p_{\ra X^0}^*}{\lra} 
 \DMZ[{R}]{X\times \Gm} \st{\Psi_f}{\lra} \DMZ[{R}]{N_Z^0}.
\]
In the case $X^0=\mo[n]$ and $Z=D_0$, we compose as previously this functor by
the non zero section of $p^0_{N_Z} : N_Z^0 \ra 
Z$ given by Proposition \ref{NDtrivial} and obtain
\[
\mc L_{D,\mo[n]}^{DM\, \Psi} : \DMZ[{R}]{\mo[n]} \lra \DMZ[{R}]{D_0}.
\]

The compatibility with mixed Tate categories follows from the following facts:
\begin{enumerate}
\item Tate objects  in  $\DMZ[{R}]{X^0}$ ; denoted by $R_{X^0}(i)$ in this
  remark; are lifted from the ones in 
  $\DMZ[{R}]{S}$. Hence their pull-back in  $\DMZ[{R}]{X^0\times \Gm}$ can be seen
  either as lifted from $\Gm_S$ or as lifted from $S$.
\item In the first case, by the compatibility of specialization functor with
  smooth morphisms \cite[Definition 3.1.1]{Ayoub6OGII} insures that 
\[
\Psi_f(p_{\ra \Gm}^*R_{\Gm}(i)) \simeq f|_{\{0\}}^*\circ \Psi_{\id_{\A^1}}(R_{\Gm}(i)).
\] 
Now, seeing the Tate objects as lifted from $S$ by $q_{\Gm}^*$, we can apply
Proposition 3.5.10 in \cite{Ayoub6OGII} which insures that 
\[
\Psi_{\id_{\A^1}}\circ q^*_{\Gm} \sim \id.
\]  
Hence the functor 
\[
\mc L_{X,Z}^{DM\, \Psi} : \DMZ[{R}]{X^0} {\lra} \DMZ[{R}]{N_Z^0}
\]
sends Tate object to Tate objects (eventually composing with the natural
isomorphic transformation). It also insures that its  composition with 
\[ 
p_{X^0} : \DMZ[{R}]{S} \lra  \DMZ[{R}]{X^0}
\]
equals $ \DMZ[{R}]{S} \lra  \DMZ[{R}]{N_Z^0}$ induced by the structural morphism
of $N_Z^0$.
\end{enumerate} 
This gives us the desired functor on mixed Tate categories
\[
\mc L_{D,\mo[n]}^{DM\, \Psi} : \DMZ[{R}]{\mo[n]} \lra \DMZ[{R}]{D_0}.
\]
\end{rem}

We come now to the more delicate aspect of tangential based point in general
situation. The general situation is the following : $X \st{p_X}{\lra} S$ a smooth
scheme with a strict normal crossing divisor $Z=\cup_{i\in I} Z_i$. %\marginote{in
%  the etale sens ?}. 
We will
denote by $Z_J$ the intersection $\cap_{i\in J} Z_j$. Note that the $Z_J$ are
also smooth over $S$. Note that in our applications; that is $X=\mob[n]$ and $Z=\partial
\mob[n]$; the $Z_J$'s are irreducible and we will assume it is the case for the
following description. If it is not the case, the description below works as
well with an extra care given to the various irreducible components of the
$Z_J$.

As previously let $X^0$ denote $X \sm Z$. Let $J$ be a subset of $I$ and let
$Z_J^0$ be the ``open stratum'' 
\[
Z_J \sm  \left( \cup_{i\in I \sm J}Z_i\cap Z_J\right). 
\]
Let $N_i$ (resp. $N_i^0$) denote the normal bundle of $Z_i$ in $X$ (resp. with
zero section removed),  $N_J$ (resp. $N_J^0$) is defined as  the fiber product
of the $N_j|_{Z_J}$ (resp. $N_j^0|_{Z_J}$) over $Z_J$ and $N_{J^0}$
(resp. $N^0_{J^0}$) its restriction to $Z_J^0$. 

We are interested in generalizing the previous situation and having a functor
\[
\mc L_{X,J} : DM(X^0) \lra DM(N_{J^0}^{0})
\]
compatible with mixed Tate categories and structural pull-back functors. We
want then to apply this functor in the case where 
 $Z_J$ is an $S$-point, that is of maximal
codimension.

First of all remark that by setting 
\[
X'=X\sm \left( \cup_{i \in I \sm J} Z_i\right) 
\]
 we can assume that $I=J$. In this case $Z_J^0$ (resp. $N_{J^0}$, $N^0_{J^0}$) is simply
 $Z_J$ (resp. $N_J$, $N_J^0$). We treat below only this situation.

Note also that in our strict normal
crossing divisor situation $N_{J}$ equals $N_{Z_J}$ the normal bundle of $Z_J$ in
$X$. Moreover locally with affine coordinates, or when $N_{Z_J}$ is trivial, one
has a  isomorphism between $N_{J}^0$
and $(\Gm_{Z_J})^{|J|}$.

\begin{lem}[{Consequence of \cite[Proposition 15.22]{OAMSpit}}]\label{SNCDfunctor}
There is a natural functor
\[
\mc L_{X,J}^{DM} : \DMZ{X^0} \lra \DMZ{N_J^0}
\]
preserving Tate objects and compatible with structural pull-back
morphisms. Hence we obtain a functor between mixed Tate categories
\[
\mc L_{X,J}^{DM} : \DMT_{/S, \Z}{(X^0)} \lra \DMT_{/S,\Z}{(N_J^0)}
\]
such that its composition with $\DMT_{/S,\Z}(S) \lra \DMT_{/S,\Z}(X^0)$ equals
the functor 
\[
\DMT_{/S, \Z}(S) \lra \DMT_{/S, \Z}(N_J^0)
\] induced by the structural morphism
of $N_J^0$.
\end{lem}
\begin{proof}[Comments on the construction]
The proof consists in obtaining the generalization of the
isomorphism \ref{Spiti*j*p*}
\[
i_{J}^*j_{X^0\,*} \MZ{X_0} \simeq p^0_{N_J\,*}\MZ{N^0_{J}}
\]
where $i_J$ denotes the regular embedding $Z_J \lra X$.

This isomorphism is obtained by taking the fiber product over $X\times \A^1$ of the
deformation $D(X,Z_j)$ (resp. $D^0(X,Z_j)$) for all $Z_j$.
% We shall write
%$D(X,J)$ and $D^0(X,J)$ these deformation. 
Then his proof goes mostly as in the case where there is only one $Z_j$ by
remarking that $\MZ{X_0}$ is the pull-back from $\MZ{S}$ by $(p_{X}\circ j_{X^0})^*$. 

As previously, the compatibility with mixed Tate object and pull-back by
structural morphism rely partly on \cite[Corollary 15.14]{OAMSpit}.

%
% It is a consequence of M. Spitzweck's functor $\mc L_{X,J} $ given in
% \cite{OAMSpit} and generalizing the functor $\mc L_{X,Z} : DM(X^0) \lra
% DM(N^0_Z)$ previously used to the case where $Z$ is a normal crossing divisor
% $Z=\cup_{i \in I} Z_i$. It gives a functor 
% \[
% \mc L_{X,J} : DM(X^0) \lra DM(N_{J^0}^{0})
% \] 
% where $N_{Z_J}^{00}$ denotes the normal bundle of $Z_J=(\cap_{j\in J} Z_j)$ in
% $X$ resticted  $Z_J^0=Z_J\sm (\cup_{i \in I\sm J} Z_i)$ minus its $0$ section.
% The construction is given in \cite{OAMSpit} after the proof of proposition
% 15.22. 
\end{proof}

\begin{rem}[on higher deformations to the normal cone and nearby cycle
  functor]\label{nearbySNCD}
As previously; especially when working with the Beilinson spectrum, that is
working with rational coefficient; one might prefer using Ayoub's nearby cycles
functor. We introduce now a higher deformation to the normal cone as presented
in \cite[Section 3.1.3]{IvorraCMIAA} following M. Rost in 
\cite[{\S  (10.6)}]{RostCGC}. Let $\mc J_j$ denote the sheaf of ideal defining
the $Z_j$'s and let $k$ denote the cardinal of $J$ (we treat only the case
$J=I$). We assume that that $J=\{1, \ldots,k\}$ as it induces an easier notation. Then 
the subalgebra of $\mc O_X[t_1,t_1^{-1}, \ldots , t_k , t_k^{-1}]$ 
\[
\mc A_{X,J}=\oplus_{(a_1, \ldots , a_k) \in \Z^k} \mc J_1^{a_1}\cdots
J_k^{a_k}t_1^{a_1} \cdots t_k^{a_k}
\]
 is quasi-coherent over $\mc O_X[t_1, \cdots, t_n]$; in the above definition, as
 previously,  $\mc J_j^{a_j}=\mc O_X$ as soon as $a_j \leqs 0$. The simultaneous
 deformation of the $Y_j$'s is defined as :
\begin{equation}
D(X;Y_1, \ldots , Y_k)=D(X;J)=\Spec\left( \mc A_{X,J} \right).
\end{equation}
Inverting the $t_j$, one obtains
\[
\mc A_{X,J} [t_1^{-1}, \ldots , t_k^{-1}] = \mc O_X[t_1,t_1^{-1}, \ldots , t_k , t_k^{-1}]
\]
and hence a canonical isomorphism between $X\times \Gm^k$ and  the restriction
of $D(X,J)$ over $ \Gm^k$ and the following diagram
\begin{equation}\label{HDefZJ}
\begin{tikzpicture}
\matrix (m) [matrix of math nodes,
 row sep=4em, column sep=1.5em,%]%, 
 text height=2.0ex,  text depth=0.25ex] 
{ D(X,Z) & & D(X,J)|_{\Gm^k} & X   \\
 & X \times \A^k & & X \times \Gm^k \\
 \A^k & & \Gm^k & \\
};
\path[->]
(m-2-4) edge node[midway,fill=white]{ } (m-2-2)
(m-1-1) edge node[auto, mathsc]{f_J} (m-3-1)
(m-1-3) edge node[left, mathsc, near start]{f_J|_{\Gm^k}} (m-3-3)
(m-3-3) edge node[auto, mathsc]{j_{\Gm^k}} (m-3-1)
(m-2-4) edge node[auto, mathsc]{p_{\ra \Gm^k}}(m-3-3)
(m-2-4) edge node[right, mathsc]{p_{\ra X}}(m-1-4)
(m-1-3) edge node[auto, mathsc]{j_{D|_{\Gm^k}}}(m-1-1)
(m-1-3) edge node[auto, mathsc]{=} (m-2-4)
(m-1-1) edge (m-2-2)
(m-2-2) edge (m-3-1)
;
%
%
%(m-1-3) edge[double] (m-2-4)
%;
\end{tikzpicture}
\end{equation}
where the square and the parallelograms are Cartesian. 

Note that the
construction is compatible with permutation of the coordinates on $\A^k$ and
permutation of the $Y_j$'s. Inverting only $t_1,\ldots, t_l$ ($l<k$), one obtains
\[
D(X,J)|_{\Gm^l\times \A^{k-l}}=\Gm^l\times D(X;Y_{l+1}, \ldots Y_k).
\]
F. Ivorra has described in \cite{IvorraCMIAA} the fiber over $t_1= \cdots
=t_l=0$. The description goes by induction on $l$. For $l=1$, $Y_{[l]}$ is
simply $Y_1$, $N_{[l]}$ is $N_1$ the normal bundle of $Y_1$ in $X$. For $l\geqs 2$,
let $Y_{[l]}$ be the intersection 
\[
Y_1\cap \cdots \cap  Y_l
\]
  and $N_{[l]}$ is the normal
bundle of $N_{[l-1]}|_{Y[l]}$ in $N_{[l-1]}$ which can be written as 
\[
N_{[l]}=N(N_{[l-1]},N_{[l-1]}|_{Y_{[l]}}).
\]
Now let $D^{[l]}$ be
the deformation
\[
D(N_{[l]};N_{[l]}|_{Y_{l+1}\cap Y_{[l]}},\ldots ,N_{[l]}|_{Y_{k}\cap Y_{[l]}}).
\]
Then one has
\[
D(X,J)|_{t_1=\cdots =t_l=0}=D^{[l]}.
\]

The fiber over $(0, \cdots, 0)$ ; that is when all the $t_j$'s are zero; is
isomorphic to $N_J$ the normal bundle of $Y_J$ in
$X$. As a last remark, the description of $\mc A_{X,J}$ shows that the
restriction of $D(X;J)$ to the diagonal (or any line going through the origin)
 is isomorphic to $D(X;Y_J)$. 

Now, as in the case of a simple deformation, one can remove the strict transform
of $Z\times \A^k$ in $D(X,J)$ and obtain an ``open deformation'' $D^0(X,J)$
whose restriction to $\Gm^k$ is simply $X^0\times \Gm^k$. Its fiber over
$(0,\ldots, 0)$ is isomorphic to $N_J^0$ and its restriction to the diagonal
$\Delta_k$ gives us 
\[
\begin{tikzpicture}
\matrix (m) [matrix of math nodes,
 row sep=4em, column sep=3em,%]%, 
 text height=2.0ex,  text depth=0.25ex] 
{ N_J^0& D^0(X,J)|_{\Delta_k}  &X^0\times \Gm & X^0   \\
  \{0\} & \A^1_S &  \Gm_S  \\
& S & \\
};
\path[->]
(m-1-1) edge node[auto, mathsc]{f_J|_{\{0\}}}(m-2-1) 
(m-1-2) edge node[auto, mathsc]{f_J}(m-2-2) 
(m-1-3) edge node[auto, mathsc]{p_{\ra \Gm}}(m-2-3) 
(m-1-1) edge node[auto, mathsc]{i_{N_Z^0}}(m-1-2)
(m-1-3) edge node[auto, mathsc]{j_{X^0\times \Gm}}(m-1-2) 
(m-2-1) edge node[auto, mathsc]{i_{\{0\}}}(m-2-2)
(m-2-3) edge node[auto, mathsc]{j_{\Gm}} (m-2-2)
(m-1-3) edge node[auto, mathsc]{p_{\ra X^0}}(m-1-4)
(m-2-1) edge node[auto, mathsc] {=}(m-3-2)
(m-2-2) edge (m-3-2)
(m-2-3) edge node[auto, mathsc] {q_{\Gm}}(m-3-2)
;
\end{tikzpicture}
\]
And we proceed as in Remark \ref{nearbyDivisor} in order to obtain a functor
\[
\mc L_{X,J}^{DM, \Psi} : \DMZ[{}]{X^0}\st{p_{\ra X0}^*}{\lra}
\DMZ[{}]{X^0\times \Gm}\st{\Psi_{f_J}}{\lra}\DMZ[{}]{N_J^0}.
\]
Compatibilities with mixed Tate categories and pull-back by structural morphisms
are as in Remark \ref{nearbyDivisor}.
\end{rem}
\begin{conj} The geometric situation given by the higher deformation
  $D(X,J)$ makes it possible to use a succession of specialization functor
  each corresponding to an $\A^1$ factor. This procedure does not depend on
  choices when considering only the mixed Tate motives categories. It agrees
  with  our construction using the diagonal. 
\end{conj}

Now, we can apply the above functor to our particular situation in order to
obtain \emph{base point at infinity} or \emph{tangential base point}. Let $n \geqs 4$.
Let $v$ be a point in $\mob[n]$ given by a closed stratum of $\partial \mob[n]$ of maximal
codimension. The stratum  $v$ is the non-empty intersection of exactly
$n-3=\dim_{S}(\mob[n])$ irreducible component of $\partial \mob[n]$ of
codimension $1$
\[
v=\bigcap_{
\substack{D \mx{ \scriptsize cl. str. of } \partial \mob[n] 
\\ \on{codim}(D)= 1
\\ v\in D}
}D =
\bigcap_{j \in J}D_j
\]
where $J=\{1,\ldots n-3\}$ corresponds to a numbering of the closed codimension
$1$ strata $D$ with $v \in D$.
The normal bundle $N_v$ of $v$ in $\mob[n]$ is trivial.
%\marginote{is it true even over arbitrary $S$?}. 
\begin{defn} 
A \emph{tangential base point}  $x_v$ of $\mo[n]$ is the choice of
a closed strata $v$ of maximal codimension in  $\partial \mob[n]$ and of a  non
zero $S$-point in $N_v^0=N_J^0$ with the notation of Lemma \ref{SNCDfunctor}
(note that in this case $N_{J^0}^0=N_J^0$ as $v=Z_J$ can not have a non empty
intersection with any other component of $\partial \mob[n]$). 
\end{defn}
 
\begin{prop}\label{motivic-base-points}
For any tangential base point $x_v$ of $\mo[n]$, there is a natural functor 
\[
 \tilde x_v^{DM, *} : \DMZ{\mo[n]} \lra \DMZ{S}
\]
sending the Tate object to the Tate object 
and hence inducing a natural
functor% \marginote{what about compact object and $\DMT_{gm}$}
\[
 \tilde x_v^* : \DMT_{/S, \Z}(\mo[n]) \lra \DMT_{/S, \Z}(S).
\]
Both functors are compatible, in the sens of Lemma \ref{SNCDfunctor}, with the
pull-back by the structural functor from $\DMZ{S}$.
\end{prop}
\begin{proof}
The boundary $\partial \mob[n]$ can be written as the union of it codimension
$1$ irreducible component
\[
\partial \mob[n]=\bigcup_{i \in I} D_i.
\]
The closed strata $v$ defines a subset $J$ of $I$ by 
\[
v= \bigcap_{j\in J} D_j.
\]
With 
\[
X=\mob[n]\sm (\cup_{i \in I \sm J}) D_i \quad \mx{and} \quad 
Z_j=D_j\sm (\cup_{i \in I  \sm J}) D_i\cap D_j  
\]
Lemma \ref{SNCDfunctor} gives us a functor
\[
\mc L_{X,J}^{DM} : \DMZ{X^0}=\DMZ{\mo[n]} \lra \DMZ{N_J^0}.
\]
 The functor $ \tilde x_v^{DM, *} $ is obtained
by composing $\mc L_{X,J}^{DM}$ with the pull-back of the $S$-point in $N_J^0$. 
In this application, we have simply $Z_J=v$  and that $Z_J^0=Z_J$.
\end{proof}

There is a canonical system of tangential based
point over $\Spec(\Z)$ on $\mo[n]$. A point $v=\cap_{j \in J}D_j  $ of this
system is given by the closed 
strata of maximal codimension of $\partial \mob[n]$. In order to choose an
$S$-point in its normal bundle, we  chose a dihedral structure $\delta$ on the
marked points  and  vertex coordinates $x_j$ corresponding to the point
$v$ (cf. \cite[Definition 2.18]{BrownMZVPMS}). Note that the chosen vertex
coordinates might differs only by the choice of the numbering. These vertex
coordinates induce a basis on $N_J$. The sum of the vector of this basis depends
only on the dihedral structure $\delta$ and is the $S$-point in $N_J^0$ attached to $v$ and
$\delta$. Changing the delta amount in introducing signs instead of taking the
sums of the basis's vector (see. \cite[{\S 2.7}]{BrownMZVPMS}). 

\begin{defn}\label{Pninftybasedpoints}Let $P_{n,\infty}$ denote the set of these
  tangential based points.  
\end{defn}
Remark that F. Brown developed his notion of based point
at infinity for the $\mo[n]$ in relation with the question of unipotent closures
and periods of the moduli space of curves in genus $0$: see  \cite{BrownMZVPMS}
Definition 3.16 and Example 3.17 and before Section 6.3.   

\begin{rem}\label{rem:Qcoeffield} The results of the above subsection and of
  section \ref{sec:m0n} 
  hold by the same arguments in a ``more classical'' motivic category as the one
  developed by Cisinski and Déglise \cite{CiDegTCM} : rational coefficients over
  a general base and Beilinson $E_{\infty}$-ring spectrum. Later on we will be
  interested in the case where the base is either a number field of the ring of
  integer of such a field with some prime inverted (see Section
  \ref{rationalcoefsetting}).  
\end{rem}
\subsection{The motivic short exact sequence}

Derived groups scheme were in particular studied by B. Toën in \cite{ToenHHCS} and
Spitzweck in \cite{DFGTMSpit}. They can be considered as spectrum of $E_{\infty}$
algebras (with a coproduct structures).
Now, using Spitzweck \cite{DFGTMSpit}, we define a derived group scheme 
associated to  the category $\DMZ{\mo[n]}$. Recall that $S=\Spec(\Z)$. Using
Theorem \ref{BSZ} and Theorem \ref{thm:M0nMTMBS}, we can 
apply directly Theorem 8.4 in \cite{SpitCP1S}.
\begin{thm}\label{thm:DMTM0n}Let $n\geqs 3$ be an integer. There is an affine derived group scheme
  $G^{\bullet}_{\Z,n}$ over $\Z$ such that 
\[
\on{Perf}(G^{\bullet}_{\Z,n}) \simeq \DMT^{gm}_{/S, \Z}(\mo[n])
\]  
where $\on{Perf}$ denotes the category of perfect representation and
$DM^{gm}_{/S, \Z}$
the full subcategory of $\DMT_{/S, \Z}(\mo[n])$ of compact objects. We shall write simply
$G^{\bullet}(\Z)$ for $G^{\bullet}_{\Z,3}$.

Moreover the structural morphism $p_n : \mo[n] \ra S=\Spec(\Z)$ induces a surjective
morphism
\[
G^{\bullet}_{\Z,n} \st{\phi_n}{\lra} G^{\bullet}_{\Z,3} \lra 0.
\]
induced by the natural pull-back $p^*_n$ at the category level. Permutation of
the marked points on $\mo[n]$ induces an action of the symmetric group on $\Gb_{\Z,n}$.

Note that similarly, we define  $G^{\bullet}_{\Z,k_1,k_2}$ associated to
$\DMT^{gm}_{/S,\Z}(\mo[k_1]\times \mo[k_2])$.
\end{thm}
Defining $\Kb_{\Z,n}$ as the kernel of $\phi_{n}$, we obtain for any $n\geqs 4$ a
short exact sequence
\begin{align}\label{motSES}
0 \lra \Kb_{\Z,n} \lra \Gb_{\Z,n} \lra \Gb(\Z). \tag{SES$_n$}
\end{align}

The compatibility property with structural morphisms in Proposition
\ref{motivic-base-points} shows that this short exact sequence is split by
any choice of a (tangential) $S$-point of $\m_{0,n}$. We restrict ourselves to
the family of tangential based points in $P_{,\infty}$.

\begin{prop}\label{prop:sectgbasepoint}
Let $x_v$ be a tangential based point $\mo[n]$ ($n\geqs 4$) in
$P_{n,\infty}$. Then $x_v$ induces a splitting
\[
\begin{tikzpicture}
\matrix (m) [matrix of math nodes,
 row sep=2em, column sep=2em,%]%, 
 text height=2.0ex,  text depth=0.25ex] 
{0 & \Kb_{\Z,n} & \Gb_{\Z,n} & \Gb(\Z) & 0  \\
};
\path[->,font=\scriptsize]
(m-1-1) edge  (m-1-2)
(m-1-2) edge  (m-1-3)
(m-1-4) edge  (m-1-5)
(m-1-3) edge  node[below] {$\phi_n$}(m-1-4)
(m-1-4) edge[bend right] node[above] {$\tilde x_v$}  (m-1-3)
;
%\path[->,font=\scriptsize, bend left]
%(m-1-2) edge node[auto] {$\otimes^L\MZ{S}$} (m-1-1);
\end{tikzpicture}
\]
\end{prop}
\begin{prop}\label{prop:sescompatiforget}
The morphisms $\psi_{n,i} : \mo[n] \lra \mo[n-1]$ forgetting the $i$-th marked
points induces a commutative diagram :
\[
\begin{tikzpicture}
\matrix (m) [matrix of math nodes,
 row sep=2em, column sep=2em,%]%, 
 text height=2.0ex,  text depth=0.25ex] 
{0 & \Kb_{\Z,n} & \Gb_{\Z,n} & \Gb(\Z) & 0  \\
0 & \Kb_{\Z,n-1} & \Gb_{\Z,n-1} & \Gb(\Z) & 0  \\
};
\path[->,font=\scriptsize]
(m-1-1) edge  (m-1-2)
(m-1-2) edge  (m-1-3)
(m-1-4) edge  (m-1-5)
(m-1-3) edge  node[below] {$\phi_n$}(m-1-4)
%(m-1-4) edge[bend right] node[above] {$x_v$}  (m-1-3)
;
\path[->,font=\scriptsize]
(m-2-1) edge  (m-2-2)
(m-2-2) edge  (m-2-3)
(m-2-4) edge  (m-2-5)
(m-2-3) edge  node[below] {$\phi_{n-1}$}(m-2-4)
(m-1-2) edge  node[left] {$\tilde \psi_{n,i}$}(m-2-2)
(m-1-3) edge  node[left] {$\psi_{n,i}$}(m-2-3)
%(m-1-4) edge[bend right] node[above] {$x_v$}  (m-1-3)
;
\path[font=\scriptsize]
(m-1-4) edge[double]  (m-2-4);
;
\end{tikzpicture}
\]
where by an abuse of notation the morphisms between derived groups are denoted
as the morphisms between schemes.
\end{prop}
\begin{proof}
The functoriality of the pull-back functors makes the equivalent of the right
square commutes between the categories $\DMT_{/S,\Z}^{gm}$.  Considering these
categories as
categories of perfect representation, that is as categories of
(co)modules over an $E_{\infty}$ algebra, this
functors induce morphisms between the affine derived group scheme
\[
 \psi_{n,i} :\Gb_{\Z,n} \lra  \Gb_{\Z,n-1}.
\]

compatible with $\phi_n$ and $\phi_{n-1}$ together with 
 an induced morphism on the kernels. 
\end{proof}

\begin{prop}\label{prop:sescompatigluing}
%\marginote{How do we have $DM(\mo[k_1]\times\mo[k_2] \sim DM \otimes DM$ or some
%equivalent on the $K_*$'s ?}
Let $D$ be an irreducible component of $\partial \mob[n]$ and $D_0$ its open
strata (cf. Section \ref{subsec:boundary}). $D_0$ is isomorphic to 
 \[
D_0\sim \mo[n_1]\times \mo[n_2]
\]
with $n_1+n_2=n+2$. 

The inclusion $i_D : \mob[n_1]\times \mob[n_2] \lra \mob[n]$ induces a gluing morphism :
\[
i_{n_1,n_2, D} : \Gb_{\Z,n_1,n_1} \lra  \Gb_{\Z,n}
\]
and 
\[
\tilde i_{n_1,n_2, D} : \Kb_{\Z,n_1,n_2} \lra  \Kb_{\Z,n}.
\]
Moreover the projections $\mo[n_1] \times \mo[n_2] \lra \mo[n_i]$ induce
morphisms 
\[
p_{n_1,n_2} : \Gb_{\Z, n_1,n_2}\lra \Gb_{\Z,n_1}\times \Gb_{\Z,n_2} 
\]
 and 
\[
\tilde p_{n_1,n_2} : \Kb_{\Z, n_1,n_2}\lra \Kb_{\Z,n_1}\times \Kb_{\Z,n_2}
\]
The above morphisms makes the diagrams below commute but the bottom line \emph{is
not} necessarily exact
\[
\begin{tikzpicture}
\matrix (m) [matrix of math nodes,
 row sep=2.5em, column sep=2.5em,%]%, 
 text height=2.0ex,  text depth=0.25ex] 
{0 &\Kb_{\Z,n} & \Gb_{\Z,n} & \Gb(\Z) &0  \\
0& \Kb_{\Z, n_1,n_2} & \Gb_{\Z,n_1,n_2} & \Gb(\Z) & 0 \\
 & \Kb_{\Z,n_1} \times \Kb_{\Z,n_2} & \Gb_{\Z,n_1}\times \Gb_{\Z,n_2} & \Gb(\Z) &   \\
};
\path[->,font=\scriptsize]
(m-1-1) edge  (m-1-2)
(m-1-2) edge  (m-1-3)
(m-1-4) edge  (m-1-5)
(m-1-3) edge  node[below] {$\phi_n$}(m-1-4)
%(m-1-4) edge[bend right] node[above] {$x_v$}  (m-1-3)
;
\path[->,font=\scriptsize]
(m-2-1) edge  (m-2-2)
(m-2-2) edge  (m-2-3)
(m-2-4) edge  (m-2-5)
(m-2-3) edge  node[below] {$\phi_{n-1}$}(m-2-4)
(m-2-2) edge  node[left] {$\tilde i_{n_1,n_2,D}$}(m-1-2)
(m-2-3) edge  node[left] {$i_{n_1,n_2,D}$}(m-1-3)
%(m-1-4) edge[bend right] node[above] {$x_v$}  (m-1-3)
%(m-3-1) edge  (m-3-2)
(m-3-2) edge  (m-3-3)
(m-3-3) edge  (m-3-4)
(m-2-2) edge  node[auto]{$p_{n_1,n_2}$}(m-3-2)
(m-2-3) edge  node[auto]{$p_{n_1,n_2}$}(m-3-3)
;
\path[->, font=\scriptsize]
(m-2-4) edge node[fill=white]{$=$}  (m-1-4)
(m-2-4) edge node[fill=white]{$=$}  (m-3-4);
%;
\end{tikzpicture}
\]
As previously, by an abuse of notation, the morphisms between derived groups are denoted
as 
the morphisms between the scheme.

\end{prop}
%\marginote{How do we have $DM(\mo[k_1]\times\mo[k_2] \sim DM \otimes DM$ or some
%equivalent on the $K_*$'s ? Do we have an iso $\Gb_{\Z,n_1}\times
%\Gb_{\Z,n_2} \simeq \Gb_{\Z,n_1}$}
\begin{proof}
First of all, Proposition \ref{divisorfunctor} gives us a functor
\[
DM(\mo[n])\lra DM(D_0) = DM(\mo[n_1]\times \mo[n_2])
\]
inducing a morphism on the groups
\[
\Gb_{\Z,n_1, n_2} \lra \Gb_{\Z,n}.
\]
Then, the projections $\mo[n_1]\times \mo[n_2] \lra \mo[n_i]$ induce
morphisms $\Gb_{\Z,n_1,n_2} \ra \Gb_{\Z,n_i}$.
\end{proof}
%\subsection[forgetfull morphism]{Compatibility with forgetfull morphism}
%\subsection[{irreducible components of $\partial \mob[n]$}]{Compatibility with 
%  irreducible component of the boundary}
\subsection{A motivic Grothendieck-Teichmüller group}
The family of the derived groups $\Kb_n$ endowed with the gluing morphism 
$\tilde i_{n_1,n_2,D}$ and forgetful morphisms $\tilde \psi_{n, i}$ is close to an operadic
structure.

We can now define the integral derived
motivic Grothendieck-Teichmüller group over $\Z$ as follows.
% \begin{defn}[motivic Grothendieck-Teichmüller ($S=\Spec(\Z)$, $\Z$ coefficent)]
%   Let $GT_{\Z}^{\bullet}(S)$ be the \marginote{derived ?} groups of
%   automorphisms ; that is of auto equivalence ; 
%   $g=(g_n)_{n\geqs 4}$ of the tower of the derived groups $(\Kb_n)_{n\geqs 4}$;
%   such that  each
%   $g_n$ is an automorphism of $K_n$ and the $g_n$ commutes with morphisms
%   $\tilde i_{n_1,n_2, D}$ and $\tilde \psi_{n,i}$
% \end{defn} 
\begin{defn}[motivic Grothendieck-Teichmüller ($S=\Spec(\Z)$, $\Z$
  coefficent)]\label{GTderb} 
  Let $GT_{\Z}^{\bullet}(S)$ be the \marginote{derived ?} groups of
  automorphisms $g$; that is of auto equivalence ; 
   of the tower of the derived groups $(\Kb_n)_{n\geqs 4}\cup
   ((\Kb_{n_1,n_2})_{n_1,n_2\geqs 4})$; 
  such that  $g$ is given by two collections of morphisms $(g_n)_{n\geqs 4}$ and
  $(g_{n_1,n_2})_{n_1,n_2 \geqs 4}$ such that each
  $g_n$ (resp. $g_{n_1,n_2}$) is an automorphism of $\Kb_n$ (resp. $\Kb_{n_1,n_2}$) and
  the $g_n$'s and the $g_{n_1,n_2}$'s commute with the action of the symmetric
  group on $\Kb_n$ and with morphisms 
  $\tilde i_{n_1,n_2, D}$, $\tilde p_{n_1,n_2,D}$  and $\tilde \psi_{n,i}$.
\end{defn} 
\begin{rem}
It seems that in \cite{Schneps-HFGMGT} it is only required that the morphism
$g_n$ preserves the image of morphisms $\tilde i_{n_1,n_2,D}$. If the approach
in \cite{Schneps-HFGMGT}
is more workable, it is not as precise in its geometric implication as the one
described here.
\end{rem}
\section{Comparison with classical motivic constructions}
\label{rationalcoefsetting}
In this section we develop how the above situation evolves in the more
classical setting of rational coefficients and when working over a number field. 
\subsection{Rational coefficients}\label{rationalcoef}
In this subsection we describe the situation with rational coefficients. The main
advantage is that the weight $t$-structure is then available and allow us to use
the tannakian formalism on the \emph{non derived} category of mixed Tate motive.

In order to work with rational coefficients one can consider the Beilinson
spectrum  $H_{B,S}$ (see \cite[Definition 13.1.2]{CiDegTCM}) and work with the
homotopy category of module over it as our derived categories of motive with $\Q$
coefficients $\DMZ[\Q]{S}$ (resp. $\DMZ[\Q]{\mo[n]}$ for $n\geqs 4$).  All the
proof of the needed statement in the previous sections work with the same
argument. Another way to work with rational coefficients is to work with the
rationalization $\MZ[\Q] S$ of M. Spitzweck's spectrum $\MZ S$. Both approaches
are equivalent 
thanks to Theorem 7.14 and Theorem 7.18 in \cite{SpitCP1S} which gives an
isomorphism
\[
\MZ[\Q]{S} \simeq H_{B,S};
\]
and more generally, using the pull-back by the structural morphisms $\MZ[\Q]{X}
\simeq H_{B,X}$. 

Tate objects $\MZ[\Q]{X}(n)$ will simply be denoted by $\Q_X(n)$ when there is no
need to insist on the spectrum they are coming from and simply by $\Q(n)$ when
it is clear enough what $X$ is. As previously, the derived category of mixed
Tate motives $\DMT_{/S,\Q}(X)$ is defined as the full triangulated subcategory
\marginote{The compactness question is not clear so far.} of  $\DMZ[\Q]{X}$
generated by Tate objects $\Q_X(n) $ for $n \in \Z$. 

Hence, from the previous section, we have a family of diagrams between derived
categories of mixed Tate motives with rational coefficients ($S=\Spec(\Z)$)
\begin{equation}\label{diagDMT}
\begin{tikzpicture}[baseline]
\matrix (m) [matrix of math nodes,
 row sep=3em, column sep=3em,%]%, 
 text height=2.0ex,  text depth=0.25ex] 
{\DMT_{/S, \Q}(\mo[n-1]) &  \DMT_{/S, \Q}(S)  \\
\DMT_{/S, \Q}(\mo[n]) &  \DMT_{/S, \Q}(S) \\
\DMT_{/S, \Q}(\mo[n_1]\times \mo[n_2]) &  \DMT_{/S, \Q}(S) \\};
\path[->]
(m-1-1) edge node[left, mathsc]{\phi_{n,i}^*} (m-2-1)
(m-1-2) edge node[mathsc, above]{\phi_{n-1}^*} (m-1-1)
(m-2-2) edge node[mathsc, above]{\phi_{n}^*} (m-2-1)
(m-3-2) edge node[mathsc, below]{\phi_{n_1,n_2}^*} (m-3-1)
(m-2-1) edge node[mathsc, left]{\mc L_{D, \mo[n]}} (m-3-1)
%;
%\path[-]
(m-1-2) edge node [right, mathsc]{=} (m-2-2) 
(m-2-2) edge node [right, mathsc]{=} (m-3-2)
(m-1-1) edge[dotted, bend right] node[mathsc, fill=white] {\tilde x_{v'}^*} (m-1-2)
(m-2-1) edge[dotted, bend right] node[mathsc, fill=white] {\tilde x_{v}^*}(m-2-2)
;
\end{tikzpicture}
\end{equation}
for any $n \geqs 4$, $n_1+n_2=n+2$ and $D$ a closed codimension $1$ strata of
$\mob[n]$. In the above diagram the tangential base points $x_v$ and $x_{v'}$ are in
$P_{n, \infty}$ and in $P_{n-1, \infty}$ respectively. Moreover $x_v$ and
$x_{v'}$ are  compatible ; that is
$\phi_{n,i}(v)=v'$ and $\on{d}\phi_{n,i}(x_v)=x_{v'}$.

In order to pass to the non derived category we need to make some comments about
the $t$-structure defined in \cite{LevTM} when working over a field (see also
\cite{KTMMLevine}. As
remarked in \cite{LEVTMFG} the arguments of \cite{LevTM} go through provided that
the Beilinson-Soulé vanishing property \eqref{BS} holds. Hence when $X$ over $S$
(and $S$) satisfies
\eqref{BS} property, we obtain a tannakian category $\MTM_{/S}(X)$ of mixed Tate
motives over $X$ as the heart of $\DMT_{/S, \Q}$ by the $t$-structure (duality
and tensor structure are inherited from the one in $\DMT$). The fiber functor is
induced  by the weight graded piece $\on{Gr}_W^*$ :
\[
 \omega : M \longmapsto \bigoplus_n \Hom(\Q_X(n),\Gr_n^W(M)).
\]

Diagram \eqref{diagDMT} is compatible with the $t$-structure and induces a
similar diagram between tannakian category of mixed Tate motive
$\MTM_{/S}(\text{--})$. The tannakian formalism \marginote{and the weight filtration ?}
allows us to identify the 
categories $\MTM_{/S}(\text{--})$ with the categories of graded representations
of graded pro-unipotent affine algebraic group \marginote{over $\Q$ ? over $\Z$
  ? $S$ ?} $G_{/S,n}$ and $G_{/S,n_1,n_2}$. We may drop the subscript $/S$ when
the base scheme is clear enough and simply write $G_n$, $G_{n_1,n_2}$. As in
the above section, $G_{/S,3}$ is denoted by $G(S)$ as it corresponds to
$\MTM_{/S}(S)$. Groups $G_{/S,n}$ are sometime refereed as motivic fundamental
groups of $\mo[n]$. However we prefer to use the expression \emph{tannakian
  groups} of $\mo[n]$ as these groups are obtained from the categories
$\MTM_{/S}(\mo[n])$ by the tannakian formalism. Hence diagram \eqref{diagDMT}
leads to a diagram of groups
\begin{equation}
\begin{tikzpicture}[
baseline={(current bounding box.center)}]
\matrix (m) [matrix of math nodes,
 row sep=2.5em, column sep=2.5em,%]%, 
 text height=2.0ex,  text depth=0.25ex] 
{ 0 & K_{n-1} & G_{n-1} & G(S) & 0 \\ 
0 &K_{n} & G_{n} & G(\Z) &0  \\
0& K_{n_1,n_2} & G_{n_1,n_2} & G(\Z) & 0 \\
 & K_{n_1} \times K_{n_2} & G_{n_1}\times G_{n_2} & G(\Z)    \\
};
\path[->]
(m-1-1) edge  (m-1-2)
(m-1-2) edge  (m-1-3)
(m-1-4) edge  (m-1-5)
(m-1-3) edge  node[below,mathsc] {\phi_{n-1}}(m-1-4)
(m-2-1) edge  (m-2-2)
(m-2-2) edge  (m-2-3)
(m-2-4) edge  (m-2-5)
(m-2-3) edge  node[below,mathsc] {\phi_n}(m-2-4)
(m-1-4) edge[dotted, bend right] node[mathsc,above]{\tilde x_{v'}^*} (m-1-3)
(m-2-4) edge[dotted, bend right] node[mathsc,above]{\tilde x_{v}^*} (m-2-3)
(m-2-3) edge  node[auto,mathsc] {\phi_{n,i}}(m-1-3)
(m-2-2) edge  node[auto,mathsc] {\tilde \phi_{n,i}}(m-1-2)
%(m-1-4) edge[bend right] node[above] {$x_v$}  (m-1-3)
%;
%\path[->,font=\scriptsize]
(m-3-1) edge  (m-3-2)
(m-3-2) edge  (m-3-3)
(m-3-4) edge  (m-3-5)
(m-3-3) edge  node[below,mathsc] {\phi_{n-1}}(m-3-4)
(m-3-2) edge  node[left,mathsc] {\tilde i_{n_1,n_2,D}}(m-2-2)
(m-3-3) edge  node[left,mathsc] {i_{n_1,n_2,D}}(m-2-3)
%(m-1-4) edge[bend right] node[above] {$x_v$}  (m-1-3)
%(m-3-1) edge  (m-3-2)
(m-4-2) edge  (m-4-3)
(m-4-3) edge  (m-4-4)
(m-3-2) edge  node[auto,mathsc]{\tilde p_{n_1,n_2}}(m-4-2)
(m-3-3) edge  node[auto,mathsc]{p_{n_1,n_2}}(m-4-3)
%;
%\path[font=\scriptsize]
(m-3-4) edge node[fill=white,inner sep=2.5pt, mathsc]{=}  (m-4-4)
(m-3-4) edge node[fill=white,inner sep=2.5pt, mathsc]{=}  (m-2-4)
(m-2-4) edge node[fill=white,inner sep=2.5pt, mathsc]{=}  (m-1-4);
%;
\end{tikzpicture}
\end{equation}
where the three first line are exact.
\begin{defn}[motivic Grothendieck-Teichmüller ($S=\Spec(\Z)$, $\Q$
  coefficent)]\label{GTderbQcoeff} 
  Let $GT^{mot}(S)$ be the  groups of
  automorphisms $g$ 
   of the tower of the  groups 
\[
(K_n)_{n\geqs 4}\cup
   ((K_{n_1,n_2})_{n_1,n_2\geqs 4});
\] 
  such that  $g$ is given by two collections of morphisms $(g_n)_{n\geqs 4}$ and
  $(g_{n_1,n_2})_{n_1,n_2 \geqs 4}$ such that each
  $g_n$ (resp. $g_{n_1,n_2}$) is an automorphism of $K_n$ (resp. $K_{n_1,n_2}$) and
  the $g_n$'s and the $g_{n_1,n_2}$'s commute with the action of the symmetric
  group on $K_n$ and with morphisms 
  $\tilde i_{n_1,n_2, D}$, $\tilde p_{n_1,n_2,D}$  and $\tilde \psi_{n,i}$.
\end{defn} 

\subsection{Working over a number field and its ring of integers}
\label{numberfiel}
Working over the integer allows us to obtain a nice and general
description. However working over a number fields would allow us to have a 
concrete description of the above groups or of their Hopf algebraic avatar (in
the tannakian formalism) in terms of algebraic cycles as described in
\cite{LEVTMFG}. However, before explaining this, we will take the opportunity to
compare the above category of mixed Tate motives over $\Z$
($\MTM_{/\Spec(\Z)}(\Spec(\Z))$) with the one defined by Goncharov and Deligne in
\cite{DG} which is defined as a subcategory of $\MTM_{/\Spec(\Q)}(\Spec(\Q))$. 

The structural morphism $ p_{\Q} : \Spec(\Q) \lra \Spec(\Z)$ induces a functor
on the derived motivic categories
\[
p_{\Q}^* : \DMZ[/\Spec(\Z),\Q]{\Spec(\Z)} \lra \DMZ[\Spec(\Z),\Q]{\Spec(\Q)}
\] 
sending Tate object to Tate objects and compatible the $t$-structure. Hence it
induces a functor between mixed Tate categories
\[
p_{\Q}^* : \MTM_{/\Spec{\Z}}(\Spec(\Z)) \lra
\MTM_{/\Spec(\Z)}(\Spec(\Q))=\MTM_{/\Spec(\Q)}(\Spec(\Q)). 
\]
Using Remark \ref{BSnumberfield}, the same holds when $\Z$ is replaced by
$\mc O_{F,\mc P}$, the ring
of $\mc P$-integers of a number field $\F$ (here $\mc P$ denotes a set of finite place
of $F$); see \cite{DG, MPMTMGon}: 
\[
p_{\F, \mc P}^* : \MTM_{/S}(S) \lra
\MTM_{/S}(\Spec(\F))=\MTM_{/\Spec(\F)}(\Spec(\F)). 
\]
where $S=\Spec(\mc O_{\F,\mc P})$.

The functor $p_{\F, \mc P}$ sends the Tate object $\Q_S(i)$ to $\Q_{\F}(i)$ and
induces on the extension groups the inclusion 
\begin{multline*}
\mc O_{\F, \mc P}^*\otimes \Q=\Ext_{\MTM_{/S}(S)}(\Q_S(0),\Q_S(1))
\lra \F^*\otimes \Q= \\\Ext_{\MTM_{/\Spec(\F)}(\Spec(\F))}(\Q_{\F}(0),\Q_{\F}(1)).
\end{multline*}
Hence it induces an equivalence between the category $\MTM_{/S}(S)$ and the
category $\MTM^{DG}(\mc O_{\F ,\mc P})$ previously defined by Deligne and
Goncharov (see \cite[{\S 1.4 and 1.7}]{DG}) as 
the sub tannakian category of $\MTM_{/\Spec(\F)}(\Spec(\F))$ such that the
coaction $\Ext(\Q_{\F}(0), \Q_{\F}(1))$ on the canonical fiber functor factors
through $\mc O_{F,\mc P}$.  
\begin{prop}There is an equivalence of category
\[
\MTM_{/S}(S)\simeq \MTM^{DG}(\mc O_{\F, \mc P}).
\]
\end{prop}

Working over a number field, that is over $S=\Spec(\F)$, Theorem
\ref{thm:M0nMTMBS} also holds and the moduli spaces over curves $\m_{0,n}$
($n\geqs 3$) have a motive in $\DMZ[/S,\Q]{S}$ and satisfy the \eqref{BS}
property. Hence Levine's works shows that the tannakian group $G_{n,\F}$ associated
to  $\MTM{/S}(\m_{0,n})$ is the spectrum of a Hopf algebra $H_n$ built from
algebraic cycles (see \cite{LEVTMFG}). More precisely, let $V_n^k(p)$ be
the $\Q$ vector space freely generated by  closed irreducible subvarieties
\[
Z \subset \m_{0,n}\times (\p^1 \sm \{1\})^{2p-k} \times \A^p 
\] 
such that the projection 
\[
\m_{0,n}\times (\p^1 \sm \{1\})^{2p-k} \times \A^p \lra 
\m_{0,n}\times (\p^1 \sm \{1\})^{2p-k} 
\]
restricted to $Z$ is dominant, flat and equidimensional of dimension $0$ (that is
quasi-finite).

The symmetric group $\Sigma_{2p-k}\rtimes (\Z/2\Z)^{2p-k}$ acts on $V_n^k(p)$ by
permutation of the $\p^1\sm\{1\}$ factors and inversion $t_i \mapsto 1/t_i$ on
the same factors. Let $Alt_{2p-k}$ be the corresponding alternating
projection. The symmetric group $\Sigma_p$ acts on $V_n^k(p)$ by permutation of
the $\A^1$ factors; let $Sym_p$ denotes the corresponding symmetric projection.

The vector space $\mc N_{n}^k(p)$ is defined as $Sym_p\circ
Alt_{2p-k}(V_n^k(p))$. For fixed $p$ they form a complex with differential
induced by intersection with faces of $(\p^1 \sm \{1\})^{2p-k}$ given by $t_i=0$
and $t_i=\infty$. Concatenation of factors and induced pull-back by the diagonal 
\[
\Delta_n : \m_{0,n} \lra \m_{0,n} \times \m_{0,n}
\]
induce a product structure on 
\[
\mc N_n=\bigoplus_p \left( \bigoplus_k \mc N_n^k(p)\right).
\] 

This endows $\mc N_n$ with the structure of differential graded commutative (and
associative) algebra for the cohomological degree $k$ in superscript. Following
M. Levine in \cite{LEVTMFG}, we obtain:
\begin{coro}\label{Gn-algcycles}
Let $H_n$ be Hopf algebra given by the $H^0$ of the (associative) bar
construction of $\mc N_n$ (see \cite{BKMTM,LEVTMFG,SouBarbase}). Then, there is
an isomorphism 
\[
G_{n} \simeq \Spec(H_n)
\]
where $G_n$ is the tannakian group associated to $\MTM_{/\Spec(\F)}(\m_{0,n})$.

Moreover in the exact sequence  
\[
\begin{tikzpicture}[
baseline={(current bounding box.center)}]
\matrix (m) [matrix of math nodes,
 row sep=2.5em, column sep=2.5em,%]%, 
 text height=2.0ex,  text depth=0.25ex] 
{ 0 & K_{n} & G_{n} & G(\Spec(\F)) & 0 \\ };
\path[->,font=\scriptsize]
(m-1-1) edge  (m-1-2)
(m-1-2) edge  (m-1-3)
(m-1-4) edge  (m-1-5)
(m-1-3) edge  node[below,mathsc] {\phi_{n}}(m-1-4)
(m-1-4) edge[dotted, bend right] node[mathsc,above]{\tilde x_{v}^*} (m-1-3)
(m-1-4) edge  (m-1-5);
\end{tikzpicture}
\]
any choice of a tangential based point $x_{v}^*$ in $P_{n, \infty}$, identifies 
$K_n$ with Deligne-Goncharov motivic fundamental group
$\pi_1^{mot}(\m_{0,n},x_{v})$. Moreover it defines an action of $G(\Spec(\F))$
on $K_n$ which is coming from the action over $\Spec(\Z)$. 
\end{coro} 
 Note that the same holds for $G_{n_1,n_2}$ and that families of base points
 $x_v$ can be chosen in a compatible way. We then obtain
\begin{coro}
The tannakian group $G_{3}=G(\Z)$ injects in $GT^{mot}(\Spec(\Z))$.
\end{coro}
\begin{proof}
The diagram defining $GT^{mot}(\Spec(\Z))$ gives a morphism
\[
G(\Z)\lra GT^{mot}(\Spec(\Z)).
\]
F. Brown's work in \cite{BrownMTMZ} shows that the action of $G(\Z)$ on $K_4$ is
faithful which implies the injectivity of the above morphisms.
\end{proof}

How to  describe explicitly the Hopf algebra $H_n$ in terms of algebraic cycles in
$\mc N_n$; hence generalizing  the construction of \cite{SouMPCC}; will be addressed in
a future work.
\section{Some open questions}
We present in this last question some reasonable but still open problems related
to the present work. They are of different types. The first type concerns
a finer understanding of (derived) groups $K_n$. The second type relates our
construction to classical approaches to the Grothendieck-Teichmüller tower. 
%Some comments about groupoid structures conclude this work.%The last type

First of all, working with $\Z$ coefficient over $\Spec(\Z)$ forces to consider
the triangulated categories $\DMT_{/\Spec(\Z),\Z}(\m_{0,n})$ and affine derived
group schemes $\Gb_{\Z,n}$. The geometric part $\Kb_{\Z,n}$ of $\Gb_{\Z,n}$ has been defined
as the kernel of the structure map
\[
\begin{tikzpicture}
\matrix (m) [matrix of math nodes,
 row sep=2em, column sep=2em,%]%, 
 text height=2.0ex,  text depth=0.25ex] 
{0 & \Kb_{\Z,n} & \Gb_{\Z,n} & \Gb(\Z) & 0  \\ 
};
\path[->,font=\scriptsize]
(m-1-1) edge  (m-1-2)
(m-1-2) edge  (m-1-3)
(m-1-4) edge  (m-1-5)
(m-1-3) edge  node[below] {$\phi_n$}(m-1-4);
\end{tikzpicture}
\] 
\begin{conj}\label{conj:Kngroup}
The affine derived group scheme $\Kb_{\Z,n}$ is an affine group scheme:
\[
\Kb_{\Z,n}=K_{\Z,n}=\Spec(R_n)
\]
where $R_n$ is a commutative Hopf algebra (not necessarily co-commutative)
defined over $\Z$. 
\end{conj}

Concerning the embedding of codimension $1$ boundary components $D\simeq
\m_{0,n_1}\times \m_{0,n_2}$ we conjecture:
\begin{conj}\label{conj:Knmiso}
The induced morphism
\[
\Kb_{\Z,n_1,n_2} \xrightarrow[p_{n_1,n_2}]{} \Kb_{\Z,n_1}\times \Kb_{\Z,n_2}
\]
is an isomorphism.
\end{conj}
A stronger version would state that 
\[
\Gb_{\Z,n_1,n_2} \xrightarrow[p_{n_1,n_2}]{} \Gb_{\Z,n_1}\times \Gb_{\Z,n_2}
\]
is an isomorphism. Conjectures \ref{conj:Kngroup} and \ref{conj:Knmiso}
would allow to consider only group automorphims of $\Kb_n$ in the definition of 
$GT_{\Z}^{\bullet}(S)$.

As we have already seen, choices  tangential bases points $x_v$ in
$P_{n, \infty}$  endows $K_n$ with action of $G(\Q)$ giving it a motivic
structure. As $GT^{mot}(\Spec(\F))$ acts on the tower of the $K_n$, it acts on
the tower of their realization. One expects:
\begin{conj} One has the following isomorphism
\[
GT^{mot}(\Spec(\F))\simeq GT \simeq \on{Aut}(\mc Real_{\textrm{Betti}}(K_*)).
\]
The weight filtration on the $K_n$ induces a weight filtration on
$GT^{mot}(\Spec(\F))$ and we denote by $GRT^{mot}(\Spec(F))$ the induced sum of
graded pieces. With these notation one has 
\[
GRT^{mot}(\Spec(\F))\simeq GRT \simeq \on{Aut}(\mc Real_{\textrm{\scriptsize De Rham}}(K_*)).
\] 
\end{conj}
Note that in the above formulas the second isomorphism is a consequence of the
work of Bar-Natan \cite{Bar-NatanAssGT} following Drinfel'd work \cite{DrinQTQH}.

%%%%%%%%%%%%%%%%%%%%%%%%%%%%%%%%%%%%%%%%%%%%%%%%%%%%%%%%%%%%%%%%%%%%
%%%%%%%%%%%%%%%%%%%%%%%%%%%%%%%%%%%%%%%%%%%%%%%%%%%%%%%%%%%%%%%%%%%%
%%%%%%%%%%%%%%%%%%%%%%%%%%%%%%%%%%%%%%%%%%%%%%%%%%%%%%%%%%%%%%%%%%%%
%%%%%%%%%%%%%%%%%%%%%%%%%%%%%%%%%%%%%%%%%%%%%%%%%%%%%%%%%%%%%%%%%%%%
%%%%%%%%%%%%%%%%%%%%%%%%%%%%%%%%%%%%%%%%%%%%%%%%%%%%%%%%%%%%%%%%%%%%
%%%%%%%%%%%%%%%%%%%%%%%%%%%%%%%%%%%%%%%%%%%%%%%%%%%%%%%%%%%%%%%%%%%%

\bibliographystyle{amsalpha}
%\nocite{SpitzweckSCVTM, KTMMLevine}
\bibliography{Gtmot}

\newcommand{\etalchar}[1]{$^{#1}$}
\providecommand{\bysame}{\leavevmode\hbox to3em{\hrulefill}\thinspace}
\providecommand{\MR}{\relax\ifhmode\unskip\space\fi MR }
\providecommand{\MRhref}[2]{%
  \href{http://www.ams.org/mathscinet-getitem?mr=#1}{#2}
}
\providecommand{\href}[2]{#2}
\begin{thebibliography}{DL{\O}{\etalchar{+}}07}

\bibitem[Ayoa]{AyoubAHGFMccnII}
Joseph Ayoub, \emph{{L}'alg{\`e}bre de {H}opf et le groupe de {G}alois
  motiviques d'un corps de caract{\'e}ristique nulle {II}}, Journal f{\"u}r die
  reine und angewandte Mathematik, to appear.

\bibitem[Ayob]{AyoubRCP}
\bysame, \emph{{U}ne version relative de la conjecture des p{\'e}riodes de
  {K}ontsevich-{Z}agier}, Journal f{\"u}r die reine und angewandte Mathematik,
  to appear.

\bibitem[Ayo07a]{Ayoub6OGI}
\bysame, \emph{Les six op\'erations de {G}rothendieck et le formalisme des
  cycles \'evanescents dans le monde motivique. {I}}, Ast\'erisque (2007),
  no.~314, x+466 pp. (2008).

\bibitem[Ayo07b]{Ayoub6OGII}
\bysame, \emph{Les six op\'erations de {G}rothendieck et le formalisme des
  cycles \'evanescents dans le monde motivique. {II}}, Ast\'erisque (2007),
  no.~315, vi+364 pp. (2008).

\bibitem[Be{\u\i}84]{BeiHRVLF}
A.~A. Be{\u\i}linson, \emph{Higher regulators and values of {$L$}-functions},
  Current problems in mathematics, {V}ol. 24, Itogi Nauki i Tekhniki, Akad.
  Nauk SSSR, Vsesoyuz. Inst. Nauchn. i Tekhn. Inform., Moscow, 1984,
  pp.~181--238.

\bibitem[BFLS99]{PSEMC}
X.~Buff, J.~Fehrenbach, P.~Lochak, and L.~Schneps, \emph{Espace de modules de
  courbes, groupes modulaires et th{\'e}orie des champs}, Panorama et
  Synth{\`e}se, no.~7, SMF, 1999.

\bibitem[BK94]{BKMTM}
Spencer Bloch and Igor Kriz, \emph{Mixed tate motives}, Anna. of Math.
  \textbf{140} (1994), no.~3, 557--605.

\bibitem[Blo86]{BlochACHKT}
Spencer Bloch, \emph{Algebraic cycles and higher {$K$}-theory}, Adv. in Math.
  \textbf{61} (1986), no.~3, 267--304.

\bibitem[BN98]{Bar-NatanAssGT}
Dror Bar-Natan, \emph{On associators and the {G}rothendieck-{T}eichmuller
  group. {I}}, Selecta Math. (N.S.) \textbf{4} (1998), no.~2, 183--212.

\bibitem[Bor74]{BorCAG}
Armand Borel, \emph{Stable real cohomology of arithmetic groups}, Ann. Sci.
  \'Ecole Norm. Sup. (4) \textbf{7} (1974), 235--272 (1975). \MR{0387496 (52
  \#8338)}

\bibitem[Bro09]{BrownMZVPMS}
Francis C.~S. Brown, \emph{Multiple zeta values and periods of moduli spaces
  {$\overline{\mathfrak M}_{0,n}$}}, Ann. Sci. \'Ec. Norm. Sup\'er. (4)
  \textbf{42} (2009), no.~3, 371--489.

\bibitem[Bro12]{BrownMTMZ}
Francis Brown, \emph{Mixed {T}ate motives over {$\mathbb Z$}}, Ann. of Math.
  (2) \textbf{175} (2012), no.~2, 949--976.

\bibitem[CD09]{CiDegTCM}
D.-C. Cisinski and F~D{\'e}glise, \emph{Triangulated categories of motives},
  http://arxiv.org/abs/0912.2110, 2009.

\bibitem[D{\'e}g08]{AGTIIDeg}
Fr{\'e}d{\'e}ric D{\'e}glise, \emph{Around the {G}ysin triangle. {II}}, Doc.
  Math. \textbf{13} (2008), 613--675.

\bibitem[DG05]{DG}
Pierre Deligne and Alexander~B. Goncharov, \emph{Groupes fondamentaux
  motiviques de {T}ate mixte}, Ann. Sci. \'Ecole Norm. Sup. (4) \textbf{38}
  (2005), no.~1, 1--56.

\bibitem[DL{\O}{\etalchar{+}}07]{MHTDLOR}
B.~I. Dundas, M.~Levine, P.~A. {\O}stv{\ae}r, O.~R{\"o}ndigs, and V.~Voevodsky,
  \emph{Motivic homotopy theory}, Universitext, Springer-Verlag, Berlin, 2007,
  Lectures from the Summer School held in Nordfjordeid, August 2002.

\bibitem[DM69]{DMISCGG}
Pierre Deligne and D.~Mumford, \emph{The irreducibility of the space of curves
  of given genus}, Pub. Math. Institut des Hautes Etudes Scientifiques (1969),
  no.~36, 75--109.

\bibitem[Dri91]{DrinQTQH}
V.~G. Drinfel{'}d, \emph{On quasitriangular quasi-{H}opf algebras and on a
  group that is closely connected with {${\on Gal}(\overline{\mbb Q}/{\mbb
  Q})$}}, translation in Leningrad Math. J. \textbf{2} (1991), no.~4, 829--860.

\bibitem[Ful98]{FultonIT}
William Fulton, \emph{Intersection theory}, second ed., Ergebnisse der
  Mathematik und ihrer Grenzgebiete. 3. Folge. A Series of Modern Surveys in
  Mathematics [Results in Mathematics and Related Areas. 3rd Series. A Series
  of Modern Surveys in Mathematics], vol.~2, Springer-Verlag, Berlin, 1998.

\bibitem[Gon01]{MPMTMGon}
A.~B. Goncharov, \emph{Multiple polylogarithms and mixed tate motives},
  www.arxiv.org/abs/math.AG/0103059, May 2001.

\bibitem[Hov01]{HoveySSSGMC}
Mark Hovey, \emph{Spectra and symmetric spectra in general model categories},
  J. Pure Appl. Algebra \textbf{165} (2001), no.~1, 63--127.

\bibitem[HS00]{Schneps-HFGMGT}
David Harbater and Leila Schneps, \emph{Fundamental groups of moduli and the
  {G}rothendieck-{T}eichm\"uller group}, Trans. Amer. Math. Soc. \textbf{352}
  (2000), no.~7, 3117--3148.

\bibitem[Ivo14]{IvorraCMIAA}
Florian Ivorra, \emph{Cycle modules and the intersection {$A_\infty$}-algebra},
  Manuscripta Math. \textbf{144} (2014), no.~1-2, 165--197.

\bibitem[Jar00]{JardMSS}
J.~F. Jardine, \emph{Motivic symmetric spectra}, Doc. Math. \textbf{5} (2000),
  445--553 (electronic).

\bibitem[Kah05]{KahnAKTACAG}
Bruno Kahn, \emph{Algebraic {$K$}-theory, algebraic cycles and arithmetic
  geometry}, Handbook of {$K$}-theory. {V}ol. 1, 2, Springer, Berlin, 2005,
  pp.~351--428.

\bibitem[Kee92]{IMSCKeel}
Sean Keel, \emph{Intersection theory of moduli space of stable {$n$}-pointed
  curves of genus zero}, Trans. Amer. Math. Soc. \textbf{330} (1992), no.~2,
  545--574.

\bibitem[Knu83]{PMSCKnud}
Finn~F. Knudsen, \emph{The projectivity of the moduli space of stable curves.
  {II}. {T}he stacks {$M\sb{g,n}$}}, Math. Scand. \textbf{52} (1983), no.~2,
  161--199.

\bibitem[KS94]{KaScha94}
Masaki Kashiwara and Pierre Schapira, \emph{Sheaves on manifolds}, Grundlehren
  der Mathematischen Wissenschaften [Fundamental Principles of Mathematical
  Sciences], vol. 292, Springer-Verlag, Berlin, 1994, With a chapter in French
  by Christian Houzel, Corrected reprint of the 1990 original.

\bibitem[Lev93]{LevTM}
M.~Levine, \emph{Tate motives and the vanishing conjecture for algebraic
  k-theory}, Algebraic {$K$}-{T}heory and {A}lgebraic {T}opology, Lake Louise
  1991 (Paul~G. Goerss and John~F. Jardine, eds.), NATO Adv. Sci. Inst. Ser C
  Math. Phys. Sci., no. 407, Kluwer Acad. Pub., Fevrier 1993, pp.~167--188.

\bibitem[Lev94]{LevBHCG}
Marc Levine, \emph{Bloch's higher {C}how groups revisited}, Ast\'erisque
  (1994), no.~226, 10, 235--320, $K$-theory (Strasbourg, 1992).

\bibitem[Lev98]{LevMM}
\bysame, \emph{Mixed motives}, no.~57, AMS, Providence, 1998.

\bibitem[Lev05]{KTMMLevine}
Marc Levine, \emph{Mixed motives}, Handbook of {$K$}-{T}heory (E.M Friedlander
  and D.R. Grayson, eds.), vol.~1, Springer-Verlag, 2005, pp.~429--535.

\bibitem[Lev10]{LEVTMFG}
Marc Levine, \emph{{Tate motives and the fundamental group.}}, Cycles, motives
  and Shimura varieties (V.~Srinivas, ed.), Tata Inst. Fund. Res. Stud. Math.,
  Tata Inst. Fund. Res., Mumbai, 2010, pp.~265--392.

\bibitem[MV99]{MorVoeHTS}
Fabien Morel and Vladimir Voevodsky, \emph{{${\mathbb A}^1$}-homotopy theory of
  schemes}, Inst. Hautes \'Etudes Sci. Publ. Math. (1999), no.~90, 45--143
  (2001).

\bibitem[Nak96]{NakamuraCUMRGT}
Hiroaki Nakamura, \emph{Coupling of universal monodromy representations of
  {G}alois-{T}eichm\"uller modular groups}, Math. Ann. \textbf{304} (1996),
  no.~1, 99--119.

\bibitem[Ros96]{RostCGC}
Markus Rost, \emph{Chow groups with coefficients}, Doc. Math. \textbf{1}
  (1996), No. 16, 319--393 (electronic).

\bibitem[Sou10]{SouMDS}
Ismael Soud{\`e}res, \emph{Motivic double shuffle}, Int. J. Number Theory
  \textbf{6} (2010), no.~2, 339--370.

\bibitem[Sou12]{SouMPCC}
\bysame, \emph{Cycle complex over {$\mathbb P^1$} minus {$3$} points : toward
  multiple zeta values cycles.}, http://arxiv.org/abs/1210.4653, 2012.

\bibitem[Sou14]{SouBarbase}
\bysame, \emph{A relative basis for mixed tate motives over the projective line
  minus three points}, http://arxiv.org/abs/1312.1849, 2014.

\bibitem[Spi01]{OAMSpit}
Markus Spitzweck, \emph{Operads, algebras and modules in model categories and
  motives}, Ph.D. thesis, Bonn Universit{\"a}t, 2001.

\bibitem[Spi05]{SpitMALM}
\bysame, \emph{Motivic approach to limit sheaves}, Geometric methods in algebra
  and number theory, Progr. Math., vol. 235, Birkh\"auser Boston, Boston, MA,
  2005, pp.~283--302.

\bibitem[Spi10]{DFGTMSpit}
Markus Spitzweck, \emph{Derived fundamental groups for tate motives}, preprint,
  2010.

\bibitem[Spi13]{SpitCP1S}
\bysame, \emph{{A} commutative {$\mathbb P^1$}-spectrum representing motivic
  cohomology over dedekind domains}, preprint, 2013.

\bibitem[SV00]{SusVoeBKCMC}
Andrei Suslin and Vladimir Voevodsky, \emph{Bloch-{K}ato conjecture and motivic
  cohomology with finite coefficients}, The arithmetic and geometry of
  algebraic cycles ({B}anff, {AB}, 1998), NATO Sci. Ser. C Math. Phys. Sci.,
  vol. 548, Kluwer Acad. Publ., Dordrecht, 2000, pp.~117--189.

\bibitem[Toe03]{ToenHHCS}
Bertrand Toen, \emph{Homotopical and higher categorical structures in algebraic
  geometry}, habilitation thesis, 2003.

\bibitem[Vez01]{VezBSSH}
Gabriele Vezzosi, \emph{Brown-{P}eterson spectra in stable {$\Bbb
  A^1$}-homotopy theory}, Rend. Sem. Mat. Univ. Padova \textbf{106} (2001),
  47--64.

\bibitem[Voe00]{Vo00}
V.~Voevodsky, \emph{Triangulated category of motives over a field}, Cycles,
  transfers, and motivic homology theories, Annals of Math. Studies, vol. 143,
  Princeton University Press., 2000.

\bibitem[Voe11]{VoeMCZ-lC}
Vladimir Voevodsky, \emph{On motivic cohomology with {$\bold
  Z/l$}-coefficients}, Ann. of Math. (2) \textbf{174} (2011), no.~1, 401--438.

\end{thebibliography}
%\bibliography{bibliographiegenerale}
\end{document}